\documentclass[hidelinks,onefignum,onetabnum]{siamart220329}

\usepackage{caption, subcaption}
\usepackage{rotating}
\usepackage{comment}
\usepackage{tikz}
\usepackage{xcolor}
\usepackage{multirow}

\renewcommand{\S}{\mathbb{S}}
\newcommand{\R}{\mathbb{R}}
\newcommand{\inprod}[2]{\left\langle #1, #2 \right\rangle}
\DeclareMathOperator{\argmin}{argmin}
\DeclareMathOperator{\diag}{diag}
\newcommand{\norm}[1]{\left\lVert #1 \right\rVert}

\newcommand{\blue}[1]{\textcolor{blue}{#1}}

\usepackage{lipsum}
\usepackage{amsfonts}
\usepackage{graphicx}
\usepackage{epstopdf}
\usepackage{algorithmic}
\ifpdf
  \DeclareGraphicsExtensions{.eps,.pdf,.png,.jpg}
\else
  \DeclareGraphicsExtensions{.eps}
\fi

\newsiamremark{remark}{Remark}
\newsiamremark{hypothesis}{Hypothesis}
\crefname{hypothesis}{Hypothesis}{Hypotheses}
\newsiamthm{claim}{Claim}

\headers{Nesterov's Accelerated Jacobi-Type Methods}{L. Liang, Q. Pang, K.-C. Toh and H. Yang}

\title{Nesterov's Accelerated Jacobi-Type Methods for Large-scale Symmetric Positive Semidefinite Linear Systems}

\author{Ling Liang\thanks{Department of Mathematics, The University of Tennessee, Knoxville, TN 37916, U.S.A. 
  (\email{liang.ling@u.nus.edu}).}
\and Qiyuan Pang\thanks{Department of Mathematics, Purdue University, IN 47907, U.S.A. 
  (\email{qpang413@gmail.com}).}
\and Kim-Chuan Toh\thanks{Department of Mathematics and Institute of Operations Research and Analytics, National University of Singapore, Singapore 119076
	(\email{mattohkc@nus.edu.sg}).}
\and Haizhao Yang\thanks{Department of Mathematics and Department of Computer Science, University of Maryland, College Park, MD 20742, U.S.A.
	(\email{hzyang@umd.edu}).}}

\usepackage{amsopn}

\ifpdf
\hypersetup{
  pdftitle={Nesterov's Accelerated Jacobi-Type Methods},
  pdfauthor={L. Liang, Q. Pang, K.-C. Toh and H. Yang}
}
\fi

\begin{document}

\maketitle

\begin{abstract}
Solving symmetric positive semidefinite linear systems is an essential task in many scientific computing problems. While Jacobi-type methods, including the classical Jacobi method and the weighted Jacobi method, exhibit simplicity in their forms and friendliness to parallelization, they are not attractive either because of the potential convergence failure or their slow convergence rate. This paper aims to showcase the possibility of improving classical Jacobi-type methods by employing Nesterov's acceleration technique that results in an accelerated Jacobi-type method with improved convergence properties. Simultaneously, it preserves the appealing features for parallel implementation. In particular, we show that the proposed method has an $O\left(t^{-2}\right)$ convergence rate in terms of objective function values of the associated convex quadratic optimization problem, where $t\geq 1$ denotes the iteration counter. To further improve the practical performance of the proposed method, we also develop and analyze a restarted variant of the method, which is shown to have an $O\left((\log_2(t))^2t^{-2}\right)$ convergence rate when the coefficient matrix is positive definite. Furthermore, we conduct appropriate numerical experiments to evaluate the efficiency of the proposed method. Our numerical results demonstrate that the proposed method outperforms the classical Jacobi-type methods and the conjugate gradient method, and shows a comparable performance as the preconditioned conjugate gradient method with a diagonal preconditioner. Finally, we develop a parallel implementation and conduct speed-up tests on some large-scale systems. Our results indicate that the proposed framework is highly scalable.
\end{abstract}



\section{Introduction}

In this paper, we consider the following symmetric positive semidefinite (SPSD) system of linear equations:
\begin{equation}
	\label{eq-linear-sys}
	Qx = b, 
\end{equation}
where $Q\in \S^n$ is a real symmetric positive semidefinite matrix, and $b\in \R^n$ is the right-hand-side vector that is assumed to be in the range space of $Q$, i.e. linear system \eqref{eq-linear-sys} is assumed to be consistent. Notice that solving the linear system \eqref{eq-linear-sys} is equivalent to solving the following unconstrained convex quadratic programming (QP) problem:
	\begin{equation}
		\label{eq-qp}
		\min_{x\in \R^n}\; f(x):=\frac{1}{2}\inprod{x}{Qx} - \inprod{b}{x}.
\end{equation}

Solving SPSD linear systems is ubiquitous for obvious reasons in various fields of science and engineering. Many physical and engineering problems are naturally described by mathematical models involving SPSD matrices. For example, the Laplacian matrix of an 
undirected graph is SPSD, and it provides a highly effective mathematical framework for analyzing the structure and behavior of networks, graphs, and systems in a wide range of disciplines \cite{chung1997spectral, ng2001spectral, shi2000normalized, pang2024distributed, pang2023spectral}. Another representative problems is to solve and analyze an elliptic partial differential equation (PDE). The stiffness matrices associated with those PDEs from finite element analysis \cite{johnson2012numerical}, solid and structural mechanics \cite{zienkiewicz2005finite}, electrostatics \cite{sadiku2001electromagnetics, tu2022linear}, and heat and mass transfer problems \cite{incropera1996fundamentals}, are typically {SPSD}. Solving linear systems arising from these systems efficiently and reliably is crucial for exploring the behavior of these PDEs and making reliable predictions. Furthermore, {SPSD} matrices frequently arise in optimization problems and machine learning tasks, both of which are foundational tools for addressing a wide range of real-world applications \cite{boyd2004convex,hinze2008optimization,jensen2024nonsmooth}. In particular, algorithmic frameworks for solving constrained and unconstrained optimization problems \cite{boyd2004convex,hinze2008optimization} often involve solving {SPSD} linear systems at their core. Efficient solution methods for such systems are essential for successful applications. 

For small to medium-sized problems, one can generally apply a direct method \cite{golub2013matrix}, including Cholesky factorization and LU decomposition methods, to obtain solutions with high precision. Moreover, it is generally robust even if the  
matrix is ill-conditioned or nearly singular. However, as the computational complexity and memory consumption of a direct method are typically $O(n^3)$ and $O(n^2)$, respectively\footnote{The complexity could be smaller when the matrix $Q$ is sparse but fill-ins would be introduced in general.}, it can become prohibitively expensive
for solving large-scale linear systems. On the other hand, iterative methods \cite{saad2003iterative,greenbaum1997iterative} require much smaller per-iteration computational complexity and are much more memory efficient, which make them particularly suitable for solving large-scale and/or sparse systems. Some notable iterative methods include Krylov subspace methods, classical Jacobi method, Gauss-Seidel method, (symmetric) successive over-relaxation method, Kaczmarz method, and their variants \cite{golub2013matrix,jacobi1846leichtes,young2014iterative,greenbaum1997iterative,karczmarz1937angenaherte,strohmer2009randomized}, to mention just a few. Iterative methods can also be terminated to achieve different levels of accuracy, which is an advantage when highly accurate solutions are not necessary. However, the convergence of iterative methods is sensitive to the spectral property (and hence the conditioning) of the coefficient matrix, and they generally require finding effective preconditioners to speed up the convergence, especially when the matrix is ill-conditioned. Furthermore, they may require too many iterations or even fail to reach a solution with a desired precision under the influence of rounding errors. But despite the shortcomings, iterative methods remain a powerful tool, and sometimes the only tool, for solving large-scale linear systems. 

For resource-intensive tasks, it is desirable to develop parallel/distributed algorithms enabling efficient processing of large data sets, improving numerical performance, and providing fault tolerance and {high applicability} in various application domains \cite{bertsekas2015parallel}. There exist highly efficient parallel implementations of direct methods (see, e.g., \cite{schenk2001pardiso,li2005overview}) that allow for significant speedup on modern parallel computing architectures. By solving large linear systems on clusters or multicore processors, one can enhance the applicability of direct methods. However, iterative methods are more attractive than direct methods in this scenario since they are less demanding in terms of memory consumption and require significantly less computational resources, especially when highly accurate solutions are not necessary. Some iterative methods, such as the Gauss-Seidel method and successive over-relaxation method, inherently follow a sequential process, which may limit their advantages in parallel or distributed computing platforms. On the other hand, certain iterative algorithms, such as the CG method, Jacobi method, and  weighted/damped Jacobi method, are more attractive since they are directly parallelizable \cite{saad2003iterative}. The parallel/distributed implementation of iterative methods for solving linear systems has been active for decades and plays a special role in today's big data era. Examples include the parallel implementation KSPCG in PETSC \cite{petsc-user-ref, petsc-efficient} for Preconditioned CG (PCG) and Parallel Genetic Algorithms (PGAs) \cite{tanese1987parallel, fogarty1991implementing, cantu1998survey, cantu2000efficient, harada2020parallel} which are parallel implementations of the Jacobi-type methods, to mention just a few. 

In this paper, we focus on Jacobi-type frameworks in order to take advantage of their simple forms and the ease of parallelization. It is well-known that the classical Jacobi method does not converge for general {SPSD} systems. It is only guaranteed to be {convergent} under strong conditions, such as strict diagonal dominance of $Q$ \cite{golub2013matrix, saad2003iterative}. A modification of the classical Jacobi method by considering a linear combination of the Jacobi update and the current iterate leads to the weighted Jacobi method. The latter method admits convergence for any symmetric positive definite system, but the weight needs to be chosen carefully \cite{saad2003iterative}. As we will also see in Section \ref{section-nesterov}, both Jacobi and weighted Jacobi methods can be viewed as special cases of the variable-metric gradient method for solving the associated convex QP problem, which requires suitable conditions for the selection of ''step sizes'' for convergence. However, both methods typically show slow convergence rates in practice (when they are convergent). 

We propose to accelerate the Jacobi-type methods by leveraging Nesterov's acceleration technique \cite{nesterov2013gradient, beck2009fast, monteiro2016adaptive, su2014differential} together with a carefully designed proximal term. The resulting algorithm has a similar updating rule as Jacobi-type methods, which allows efficient and simple parallel implementation, and it is provably convergent under mild conditions. In particular, the proposed method is shown to be {convergent} with an $O\left(\frac{1}{t^2}\right)$ convergence rate in terms of objective function values of the associated convex QP problem, where $t\geq 1$ denotes the iteration counter. Moreover, existing literature also suggests that a restarted scheme could improve the practical performance of the proposed method significantly. Hence, we also analyze the convergence properties of a restarted version of the algorithm and establish an $O\left(\frac{(\log_2 t)^2}{t^2}\right)$ convergence rate. Finally, our numerical experiments confirm the efficiency of the proposed algorithm and the scalability of its parallel implementation. Specifically, when applied to strictly diagonal dominant systems, our accelerated algorithm empirically outperforms both the classical Jacobi method and 
weighted Jacobi method with  optimal weight. Furthermore, in numerical tests involving Laplacian systems and general {symmetric positive definite} systems, our method demonstrates better empirical performance compared to the classical CG method, and shows comparable results as the preconditioned CG method using a simple diagonal preconditioner. Additionally, the speed-up tests for the parallel implementation of our method exhibit promising scalability. These combined theoretical and numerical results provide compelling evidence for the potential applicability of our algorithm in a wide range of real-world applications.

The rest of the paper is organized as follows. In Section \ref{section-nesterov}, we first review the classical Jacobi method and  weighted Jacobi method, which are special cases of the gradient method. Then, we attempt to accelerate and analyze the Jacobi-type methods by the Nesterov's acceleration technique in the same section. We also consider a restarted variant of the proposed method and analyze its convergence properties in the same section. In Section \ref{section-proximal-term}, we show how to select the proximal term in the proposed framework in order to preserve the simple updating rule as in the classical Jacobi methods. To evaluate the practical performance of the proposed methods, we conduct a set of numerical experiments in Section \ref{section-experiments}. Finally, we conclude our paper in Section \ref{section-conclusions}.

\textbf{Notation.}
The identity matrix in $\R^{n\times n}$ is denoted as $I_n$. For any matrix $A\in \R^{n\times n}$, suppose that its eigenvalues are real. We use $\lambda_{\textrm{max}}(A)$ and $\lambda_{\textrm{min}}(A)$ to denote its maximal and minimal eigenvalues, respectively. The spectral radius of the matrix $A$, i.e., the maximum of the absolute values of its eigenvalues, is denoted by $\rho(A)$. Suppose that $A\in \S_+^n$ (the set of symmetric positive semidefinite matrices), we denote $\norm{x}_A:=\sqrt{\inprod{x}{Ax}}$, for any $x\in \R^n$.

\section{The Nesterov's accelerated method} \label{section-nesterov}

In this section, we first revisit the Jacobi-type algorithmic frameworks and show that they are special cases of the general variable-metric gradient methods for solving the associated unconstrained convex QP 
problem. Hence, the convergence results from the optimization perspectives can be directly applied. Then, we consider the Nesterov's accelerated gradient method with provable convergence. We also present the convergence results of a restarted version of the former algorithm. We should mention that the analysis for the Nesterov's accelerated gradient method and its restarted variants is well-developed and mature in the literature; see e.g., \cite{bonnans1995family, beck2009fast, nesterov2013gradient, ito2017unified, li2019block, o2015adaptive, yang2023inexact, lee2022escaping, liang2022qppal} and references therein. However, existing analysis is mainly for a general convex optimization problem, whereas the focus of the present paper is an unconstrained convex QP problem. To make the paper self-contained, we also provide a concise analysis here.

To begin, let $Q$ be decomposed as $Q = D + L + L^T$ where $D \in \S^n$ and $L\in \R^{n\times n}$ capture the diagonal components and strictly lower triangular part of $Q$, respectively. Then, the classical Jacobi method performs the following update:
\begin{equation}
	\label{eq-jacobi}
	x^{t+1} = D^{-1}(b - (L+L^T)x^t) = x^t + D^{-1}(b - Qx^t),\quad \forall \; t\geq 0,
\end{equation}
where $t$ is the iteration counter. The standard condition based on the fixed-point theory for the Jacobi method to be convergent is $\rho(D^{-1}(L+L^T)) < 1$ where $\rho(\cdot)$ denotes the spectral radius \cite{golub2013matrix}. As a consequence, the Jacobi method is not always convergent for general SPSD systems given an
arbitrary initial point, as mentioned before. In order to derive convergent Jacobi-type methods, one may consider a linear combination between the current point $x^t$ and the point given by the Jacobi method:
\begin{equation}
	\label{eq-weighted-jacobi}
	x^{t+1} = \omega D^{-1}(b - (L+L^T)x^t) + (1-\omega) x^t = x^t + \omega D^{-1}(b-Qx^t), \quad \forall \; t\geq 0,
\end{equation}
where $\omega \in \left(0, \frac{2}{\lambda_{\textrm{max}}(D^{-1}Q)}\right) $. The above scheme is usually called the weighted Jacobi method, and it is convergent for any given initial point \cite{saad2003iterative} when solving {symmetric positive definite} systems. We can derive the optimal weight, denoted as $\omega_{\textrm{opt}}$ by minimizing the spectral radius $\rho(I-\omega D^{-1}Q)$ and it is given by 
\[
\omega_{\textrm{opt}} = \frac{2}{\lambda_{\textrm{min}}(D^{-1}Q) + \lambda_{\textrm{max}}(D^{-1}Q)}\in \left(0, \frac{2}{\lambda_{\textrm{max}}(D^{-1}Q)}\right).
\]
However, when $n$ is large, estimating $\lambda_{\textrm{max}}(D^{-1}Q)$ could be nontrivial and costly which limits the applicability of the algorithm in practice \cite{kimmel2019extensions}. Traditionally the weight is estimated  through trial and error \cite{saad2003iterative} and more recently through auto-tuning techniques \cite{imakura2012auto}.

{With respect to the QP problem \eqref{eq-qp},}
both updating rules in \eqref{eq-jacobi} and \eqref{eq-weighted-jacobi} can be reformulated as 
\begin{equation}
	\label{eq-gd}
	x^{t+1} = \argmin\left\{f(x^t)+ \inprod{\nabla f(x^t)}{x - x^t} + \frac{1}{2}\norm{x - x^{\blue{t}}}_{S}^2\;:\;x\in\R^n\right\},\quad \forall \; t\geq 0,
\end{equation}
where $S:=D$ for \eqref{eq-jacobi} and $S:=\frac{1}{\omega}D$ for \eqref{eq-weighted-jacobi}, provided that $D\in \S_{++}^n$ and $\omega>0$. Here, the term $\frac{1}{2}\norm{x - x^t}_S^2$ is usually called the proximal term. Note that the traditional gradient method corresponds to the setting $S = \frac{1}{\eta} I_n$, where $\eta > 0$ is the step size. Therefore, the above updating rule is a special case of the gradient method by changing the standard Euclidean metric to the metric induced by $S$. Hence, these types of methods are also known as {preconditioned} gradient methods. However, the gradient method specified by \eqref{eq-gd} is not always {convergent}. In other words, one needs to choose the matrix $S$ appropriately for convergence. For a given $S \in \S_{++}^n$, observe that the updating rule in \eqref{eq-gd} has the following form:
\[
x^{t+1} = (I_n - S^{-1}Q)x^t + S^{-1}b,\quad \forall\; t\geq 0.
\]
Clearly, a standard sufficient condition for the above scheme to be convergent is $\rho(I_n - S^{-1}Q) < 1$ \cite{saad2003iterative}. Since $S^{-1}Q$ is similar to the SPSD matrix $S^{-1/2}QS^{-1/2}$, we see that the latter condition is equivalent to 
\[
\lambda_{\textrm{max}}(S^{-1}Q) =\lambda_{\textrm{max}}(S^{-1/2}QS^{-1/2}) < 2,\quad \lambda_{\textrm{min}}(S^{-1}Q) =\lambda_{\textrm{min}}(S^{-1/2}QS^{-1/2}) > 0. 
\]
When $Q\in \S_{++}^n$, these conditions cover the sufficient conditions for the Jacobi and weighted Jacobi methods, respectively. Moreover, for the special case when $Q\in \S_{++}^n$ and $S = \frac{1}{\eta}I_n$, the above conditions reduce to $0 < \eta < \frac{2}{\lambda_{\textrm{max}}(Q)}$. If we set $\eta = \frac{2}{\lambda_{\textrm{max}}(Q) + \lambda_{\textrm{min}}(Q)}$, then it can be verified that 
\begin{align*}
	\norm{x^t - x^*} \leq &\; \left(\frac{\lambda_{\textrm{max}}(Q) - \lambda_{\textrm{min}}(Q)}{\lambda_{\textrm{max}}(Q)+\lambda_{\textrm{min}}(Q)}\right)^t\norm{x^0 - x^*}, \\
	f(x^t) - f(x^*) \leq &\; \frac{\lambda_{\textrm{max}}(Q)}{2}\left(\frac{\lambda_{\textrm{max}}(Q) - \lambda_{\textrm{min}}(Q)}{\lambda_{\textrm{max}}(Q)+\lambda_{\textrm{min}}(Q)}\right)^{2t}\norm{x^0 - x^*}^2,
\end{align*}
where $x^*$ is the unique solution of \eqref{eq-linear-sys}. The above convergence results coincide with those classical results in the field of optimization; see, e.g., \cite[Theorem 2.1.15]{nesterov2018lectures}. Unfortunately, when the condition number of $Q$, namely $\frac{\lambda_{\textrm{max}}(Q)}{\lambda_{\textrm{min}}(Q)}$, is large, the convergence rate could be rather slow. Moreover, when $Q$ is only SPSD, the above results are no longer applicable. In this case, we can still rely on the convergence results for convex optimization. For instance, suppose that $Q\in \S_{+}^n$ and $S = \frac{1}{\eta}I_n$, and $\eta = \frac{1}{\lambda_{\textrm{max}}(Q)}$, then 
\[
f(x^t) - f(x^*) \leq \frac{2\lambda_{\textrm{max}}(Q)\norm{x^0 - x^*}^2}{t+4},\quad \forall\; t\geq 0,
\]
which establishes an $O\left(\frac{1}{t}\right)$ convergence rate in terms of the objective value gap; see \cite[Corollary 2.1.2]{nesterov2018lectures} for more details. Note that the step size $\eta$ depends on $\lambda_{\textrm{max}}(Q)$ which requires additional computational cost to estimate.

The above discussion argues that the gradient method (which covers the Jacobi and weighted Jacobi methods) is generally a slow method and could be impractical. In order to speed up the gradient method, it is commonly recommended to apply the Nesterov's acceleration technique \cite{nesterov2013gradient, beck2009fast}. The template for the Nesterov's accelerated method for solving the minimization problem \eqref{eq-qp} is given in Algorithm \ref{alg-apg}. 

\begin{algorithm}[htb!]
	\begin{algorithmic}[1]
		\STATE \textbf{Input:} Initial points $x^{0} = y^1 \in \R^n$, $\alpha_0 = 0,\;\alpha_1 = 1$ and $S\in \S_+^n$.
		\FOR{$t\geq 1$}
		\STATE $x^{t} = \argmin\left\{\frac{1}{2}\inprod{x}{Qx} - \inprod{b}{x} + \frac{1}{2}\norm{x - y^t}_S^2\;:\; x\in \R^n\right\}$.
		\STATE $\alpha_{t+1} = \frac{1+\sqrt{1+4\alpha_t^2}}{2}$.
		\STATE $y^{t+1} = \left(1+\frac{\alpha_t - 1}{\alpha_{t+1}}\right)x^{t} - \frac{\alpha_t - 1}{\alpha_{t+1}}x^{t-1}$.
		\ENDFOR 
		\STATE \textbf{Output:} $x^{t}$
	\end{algorithmic}
	\caption{The Nesterov's accelerated method.}
	\label{alg-apg}
\end{algorithm}

The Nesterov's accelerated framework is recognized to be one of the most influential and efficient first-order methods in the literature. By the description of Algorithm \ref{alg-apg}, we see that the algorithm is easily applicable and simple to implement. Moreover, there are several major differences between Algorithm \ref{alg-apg} and the gradient method \eqref{eq-gd}. First, instead of using a linear approximation, i.e., $f(x^t) + \inprod{\nabla f(x^t)}{x - x^t}$, of the objective function $f$ at the current iterate $x^t$, we keep using the function $f$ that is in fact an exact second-order approximation. Second, the matrix $S$ is no longer required to be {positive definite}. Instead, only $S\in \S_+^n$ is required. Finally, the proximal term is changed to $\frac{1}{2}\norm{x - y^t}_S^2$ where $y^t$ depends on $x^{t-1}$ and $x^{t-2}$ which essentially adds a momentum term to the update of the gradient method \eqref{eq-gd} to achieve acceleration. The above features are some of the main reasons to achieve accelerated convergence rate, as we will see shortly.  

\subsection{Convergence properties}
We next analyze the convergence of the Nesterov's accelerated method, described in Algorithm \ref{alg-apg}. For simplicity, we denote 
\[
q(x,y):= f(x) + \frac{1}{2}\norm{x - y}^2_S,\quad \forall\; x,y\in \R^n.
\]
Clearly, since we require $S\in \S_+^n$, it holds that $f(x) \leq q(x,y)$ for any $x,\,y\in\R^n$. Moreover, since $b$ belongs to the range space of $Q$, $f^*:=\min_{x\in\R^n} f(x)$ is finite and attainable at a certain optimal solution $x^*$.

By the fact that $f(x)\leq q(x,y)$ for any $x,y\in\R^n$, we derive the following lemma that is useful to evaluate the function value of $f$. Note that the lemma is a simplified version of \cite[Lemma 2.3]{beck2009fast} which can be proven in a straightforward manner.
\begin{lemma} \label{lemma-diff-f}
	For any $x\in \R^n$, it holds that 
	\[
	f(x) - f(x^{t}) \geq \frac{1}{2}\inprod{x^{t} - y^t}{S(x^{t} - y^t)} + \inprod{y^t - x}{S(x^{t} - y^t)},\quad t\geq 1.
	\]
\end{lemma}
\begin{proof}
	First, we note that by the optimality condition for $x^t$ in Line 3 of Algorithm
	\ref{alg-apg}, we have that $Qx^t-b+S(x^t-y^t)=0.$
	Then, from the fact that $f(x^{t}) \leq q(x^{t}, y^t)$, we see that 
	\begin{align*}
		f(x) - f(x^{t}) \geq 
        &\; \frac{1}{2}\inprod{x}{Qx} - \inprod{b}{x} - \frac{1}{2}\inprod{x^{t}}{Qx^{t}} + \inprod{b}{x^{t}} - \frac{1}{2}\inprod{x^{t} - y^t}{S(x^{t} - y^t)} \\
		\geq &\; \inprod{Qx^{t}-b}{x - x^{t}} - \frac{1}{2}\inprod{x^{t} - y^t}{S(x^{t} - y^t)} \\
		= &\; -\inprod{S(x^{t} - y^t)}{x - x^{t}} - \frac{1}{2}\inprod{x^{t} - y^t}{S(x^{t} - y^t)} \\
		= &\; \inprod{S(x^{t} - y^t)}{x^{t}-x} - \frac{1}{2}\inprod{x^{t} - y^t}{S(x^{t} - y^t)} \\
		= &\; \frac{1}{2}\inprod{x^{t} - y^t}{S(x^{t} - y^t)} + \inprod{y^t - x}{S(x^{t} - y^t)}.
	\end{align*}
	This completes the proof.
\end{proof}

Then by using Lemma \ref{lemma-diff-f}, we can establish the following recursive descent properties, which then imply the convergence of Algorithm \ref{alg-apg}. 
\begin{lemma}\label{lemma-recursive}
	Let $w^{t}:= \alpha_{t} x^{t}  - (\alpha_{t} - 1)x^{t-1} - x^*$. Then it holds that
	\[
	\alpha_{t}^2(f(x^{t})-f^*) + \frac{1}{2}\inprod{w^{t}}{Sw^{t}} \leq  \alpha_{t-1}^2(f(x^{t-1})-f^*) + \frac{1}{2}\inprod{w^{t-1}}{Sw^{t-1}}, \quad \forall t\geq 1.
	\]
\end{lemma}
\begin{proof}
	Applying Lemma \ref{lemma-diff-f} for $x = x^{t-1}$ and $x = x^*$, respectively, we get 
	\begin{align*}
		f(x^{t-1}) - f(x^{t})  \geq &\; \frac{1}{2}\inprod{x^{t} - y^t}{S(x^{t} - y^t)} + \inprod{y^t - x^{t-1}}{S(x^{t} - y^t)}, \\
		f(x^{*}) - f(x^{t})  \geq &\; \frac{1}{2}\inprod{x^{t} - y^t}{S(x^{t} - y^t)} + \inprod{y^t - x^*}{S(x^{t} - y^t)}.
	\end{align*}
	Multiplying $\alpha_{t} - 1$ (notice that $\alpha_t \geq 1$ for $t\geq 1$) on both sides of the first inequality and adding the result to the second one yields that
	\begin{align*}
		&\; (\alpha_{t}-1)(f(x^{t-1}) - f^*) - \alpha_{t} (f(x^{t})-f^*) \\
		\geq &\;\frac{\alpha_{t}}{2}\inprod{x^{t} - y^t}{S(x^{t} - y^t)} +  (\alpha_{t}-1)\inprod{y^t - x^{t-1}}{S(x^{t} - y^t)} + \inprod{y^t - x^*}{S(x^{t} - y^t)} \\
		= &\; \frac{\alpha_{t}}{2}\inprod{x^{t} - y^t}{S(x^{t} - y^t)} + \inprod{\alpha_{t}y^t - (\alpha_{t}-1)x^{t-1} - x^*}{S(x^{t} - y^t)}
	\end{align*}
	where we have used the facts that $f(x) - f^*\geq 0$ for $x\in \R^n$, for $t\geq 1$. Then, by multiplying $\alpha_{t}$ 
	on both side of the last inequality {and using the fact that $\alpha_{t-1}^2 = \alpha_{t}(\alpha_{t} - 1)$}, we obtain that 
	\begin{align*}
		&\; \alpha_{t-1}^2 (f(x^{t-1}) - f^*) - \alpha_{t}^2(f(x^{t}) - f^*) \\
		\geq &\; \frac{\alpha_{t}^2}{2}\inprod{x^{t} - y^t}{S(x^t - y^t)} + \alpha_{t}\inprod{\alpha_{t}y^t - (\alpha_{t}-1)x^{t-1} - x^*}{S(x^{t} - y^t)} \\
		= &\; \frac{1}{2}\inprod{\alpha_{t}x^{t} - \alpha_{t}y^t}{S(\alpha_{t}x^{t} - \alpha_{t}y^t)}  + \inprod{\alpha_{t}y^t - (\alpha_{t}-1)x^{t-1} - x^*}{S(\alpha_{t}x^{t} - \alpha_{t}y^t)} \\
		= &\; \frac{1}{2}\inprod{\alpha_{t}x^{t} - (\alpha_{t}-1)x^{t-1} - x^* {- w^{t-1}}}{S(\alpha_{t}x^{t} - (\alpha_{t}-1)x^{t-1} - x^* {- w^{t-1}})} \\
		&\; {+} \inprod{{w^{t-1}}}{S({\alpha_{t}x^t} - (\alpha_{t}-1)x^{t-1} - x^* {- w^{t-1}})} \\
		= &\; \frac{1}{2}\inprod{w^{t}}{Sw^{t}} - \frac{1}{2}\inprod{w^{t-1}}{Sw^{t-1}}
	\end{align*}
	where in the {second} equality we used the fact that $\alpha_ty^t - (\alpha_t-1)x^{t-1} - x^* = w^{t-1}$, for $t\geq 1$, which can be deduced directly from Line 4 and Line 5 in Algorithm \ref{alg-apg}, {and in the third equality we substituted the definition of $w^t:=\alpha_{t} x^{t}  - (\alpha_{t} - 1)x^{t-1} - x^*$, for $t\geq 1$, as given in the statement of the lemma}.
\end{proof}

Now, we shall present the main convergence properties of the Algorithm \ref{alg-apg}.  
\begin{theorem} \label{thm-convergence}
	Let $\{x^t\}$ be the sequence generated by Algorithm \ref{alg-apg}, then it holds that
	\[
	0\leq f(x^t) - f^* \leq \frac{2\norm{x^0 - x^*}_S^2}{(t+1)^2},\quad t\geq 1.
	\]
\end{theorem}
\begin{proof}
	By the recursive relation stated in Lemma \ref{lemma-recursive}, we see that 
	\[
	\alpha_{t}^2(f(x^{t})-f^*) + \frac{1}{2}\inprod{w^{t}}{Sw^{t}} \leq \alpha_{0}^2(f(x^{0})-f^*) + \frac{1}{2}\inprod{w^{0}}{Sw^{0}} = \frac{1}{2}\inprod{w^{0}}{Sw^{0}},\quad t\geq 1.
	\]
    
    We next claim that $\alpha_t \geq (t+1)/2$ for all $t\geq 1$. We prove this by mathematical induction. For $t = 1$, we see that $\alpha_1 = (1+\sqrt{1+4\cdot 0})/2 = 1 \geq (1+1)/2$. Hence, the statement is true for $t = 1$. Suppose that the statement is true for $t$, i.e., $\alpha_t\geq (t+1)/2$. Then, it holds that 
    \[
        \alpha_{t+1} = \frac{1+\sqrt{1+4\alpha_t^2}}{2} \geq \frac{1+\sqrt{1+(t+1)^2}}{2} >  \frac{1+\sqrt{(t+1)^2}}{2} = \frac{t+2}{2}.
    \]
    Therefore, the statement is true for $t+1$, which implies that the statement is true for all $t\geq 1$. From this, we see that 
    \[
        0\leq f(x^t) - f^* \leq \frac{\inprod{w^{0}}{Sw^{0}}}{2\alpha^2_t} \leq \frac{2\norm{x^0 - x^*}_S^2}{(t+1)^2},\quad t\geq 1.
    \]   
    Hence, the proof is completed.
\end{proof}

Theorem \ref{thm-convergence} gives an improved $O\left(\frac{1}{t^2}\right)$ convergence rate that is obviously better than the traditional gradient method, whose worst-case convergence rate is $O\left(\frac{1}{t}\right)$. However, careful readers may observe that the generated objective values $\{f(x^t)\}$ are not necessarily monotonically decreasing. This is reasonable since Theorem \ref{thm-convergence} only provides a sequence of upper bounds for the objective values. {Since the objective function is only assume to be convex, the convergence of the objective values does not imply the convergence of the sequence $\{x^t\}$ to an optimal solution. A certain modification to the Nesterov's accelerated method can be applied to guarantee the convergence of $\{x^t\}$; see e.g., \cite{chambolle2015convergence}. Furthermore, it is worth noticing that, when $Q\in \mathbb{S}_{++}^n$, linear convergence can be obtained for a modified Nesterov's accelerated method \cite[Example 3.7]{valkonen2020inertial}.}

\subsection{An adaptive restarting strategy}\label{section-restart}

While Algorithm \ref{alg-apg} admits appealing convergence properties, there exist various techniques that can further improve the practical performance of the algorithm empirically; see e.g., \cite{ito2017unified,o2015adaptive} and the references therein. In this subsection, we introduce an adaptive restarting scheme and analyze its convergence. 

The template of our restarted algorithm is described in Algorithm \ref{alg-apg-re}. Here, we briefly explain the ideas of the proposed restarting scheme. Observe that the updating rule of $y^{t+1}$ reads $y^{t+1} = x^{t} + \frac{\alpha_t - 1}{\alpha_{t+1}}(x^t - x^{t-1})$, which amounts to adding the momentum $\frac{\alpha_t - 1}{\alpha_{t+1}}(x^t - x^{t-1})$ to $x^t$. The significance of this momentum $\frac{\alpha_t - 1}{\alpha_{t+1}} \in [0,1]$ approaches 1 quickly due to the fact that $\alpha_t \geq \frac{t+1}{2}$, for $t\geq 0$. This can cause the generated iterates $x^t$ and $y^t$ to oscillate (i.e., overshoot) near the optimal solution $x^*$, which further slows down the convergence speed of the algorithm, based on our empirical observation. To address this concern, an adaptive restarting strategy that intermittently resets the momentum's significance back to zero, utilizing the point generated from the preceding iteration as the new starting point (by setting $\alpha_{t+1} = 1$, $\alpha_t = 0$, $y^{t+1} = x^{t-1}$, and $x^t = x^{t-1}$). There remains one critical issue to be addressed: 
\begin{center}
    \textit{What is the right time to perform restarting?} 
\end{center}

There is no affirmative answer to this question, and the approach we employed is somewhat heuristic in nature. On the one hand, when the monitored condition {in Line 4 of Algorithm \ref{alg-apg-re}, i.e.,} $\inprod{Qy^t - b}{x^t - x^{t-1}} \geq  0$ is satisfied, we see that the angle between $\nabla f (y^t)$ and the direction of motion $x^t - x^{t-1}$ is acute. Intuitively, $x^t$ seems to make the objective function increase. So this condition can be a heuristic signal for us to trigger a restart. On the other hand, we do not want to perform the restarting too frequently in order to allow the algorithm to make significant progress. Moreover, oscillation may only be moderate, and a restart could destroy the fast convergence speed. Hence, similar to \cite{ito2017unified}, we set a prohibition period and allow the algorithm to restart only when the number of iterations between the $(\ell-1)$-th restart and the $\ell$-th restart is larger than $K_\ell$, where $K_\ell$ is updated as $K_{\ell+1} = 2K_\ell$ for $\ell\geq 0$ and $K_0\geq 2$ is specified by the user. 

\begin{algorithm}[htb!]
	\begin{algorithmic}[1]
		\STATE \textbf{Input:} Initial points $x^{0} = y^1 \in \R^n$, $\alpha_0 = 0,\;\alpha_1 = 1$, $S\in \S_+^n$, $\ell = 0$, $K_{\ell}\geq 2$ and $K_{\textrm{re}} = 0$.
		\FOR{$t\geq 1$}
		\STATE $x^{t} = \argmin\left\{\frac{1}{2}\inprod{x}{Qx} - \inprod{b}{x} + \frac{1}{2}\norm{x - y^t}_S^2\right\}$.
		\IF{$t > K_{\textrm{re}} + K_{\ell}$ and $\inprod{Qy^t - b}{x^t - x^{t-1}} \geq  0$}
		\STATE $K_{\textrm{re}} = t$.
		\STATE $K_{\ell+1} = 2K_{\ell}$.
		\STATE $\ell = \ell + 1$.
		\STATE $\alpha_{t+1} = 1$, $\alpha_t = 0$.
		\STATE $y^{t+1} = x^{t-1}$, $x^t = x^{t-1}$.
		\ELSE
		\STATE $\alpha_{t+1} = \frac{1+\sqrt{1+4\alpha_t^2}}{2}$.
		\STATE $y^{t+1} = \left(1+\frac{\alpha_t - 1}{\alpha_{t+1}}\right)x^{t} - \frac{\alpha_t - 1}{\alpha_{t+1}}x^{t-1}$.
		\ENDIF 
		\ENDFOR 
		\STATE \textbf{Output:} $x^{t}$
	\end{algorithmic}
	\caption{The accelerated gradient method with restarting strategy.}
	\label{alg-apg-re}
\end{algorithm}

Now, we can proceed to analyze the convergence of the restarted algorithm. We first introduce some notation. Let $\widetilde{K}_\ell$ denotes the number of iterations taken between the $\ell$-th and the $(\ell + 1)$-th restart, excluding the iteration at which the restart takes place. With this notation, the generated sequence $\{x^t\}$ can be denoted as $\{x^{\ell, k_\ell}\}$ where 
{$t = \sum_{j=0}^{\ell-1}\widetilde{K}_j + k_\ell$} and $0\leq k_\ell \leq \widetilde{K}_\ell$. Then, by the description of Algorithm \ref{alg-apg-re}, we can see that $f(x^{\ell, 0}) = f(x^{\ell-1, \widetilde{K}_{\ell-1}})$ for all $\ell\geq 1$. See Figure \ref{fig-restart} for an illustration of this notation used in our restarting strategy.

\tikzset{every picture/.style={line width=0.75pt}} 
\begin{figure}[htb!]
	\centering    
	\resizebox{0.95\textwidth}{!}{
		\begin{tikzpicture}[x=0.75pt,y=0.75pt,yscale=-1,xscale=1]
			
			\draw    (100,107) -- (661,106.5) ;
			\draw [shift={(663,106.5)}, rotate = 179.95] [color={rgb, 255:red, 0; green, 0; blue, 0 }  ][line width=0.75]    (21.86,-6.58) .. controls (13.9,-2.79) and (6.61,-0.6) .. (0,0) .. controls (6.61,0.6) and (13.9,2.79) .. (21.86,6.58)   ;
			\draw  [dash pattern={on 4.5pt off 4.5pt}]  (98,46.5) -- (99,129.5) ;
			\draw  [dash pattern={on 4.5pt off 4.5pt}]  (167,47.5) -- (168,126) ;
			\draw    (103,97.02) -- (140,97.48) ;
			\draw [shift={(142,97.5)}, rotate = 180.7] [color={rgb, 255:red, 0; green, 0; blue, 0 }  ][line width=0.75]    (10.93,-3.29) .. controls (6.95,-1.4) and (3.31,-0.3) .. (0,0) .. controls (3.31,0.3) and (6.95,1.4) .. (10.93,3.29)   ;
			\draw [shift={(101,97)}, rotate = 0.7] [color={rgb, 255:red, 0; green, 0; blue, 0 }  ][line width=0.75]    (10.93,-3.29) .. controls (6.95,-1.4) and (3.31,-0.3) .. (0,0) .. controls (3.31,0.3) and (6.95,1.4) .. (10.93,3.29)   ;
			\draw    (101,62.04) -- (168,63.46) ;
			\draw [shift={(170,63.5)}, rotate = 181.21] [color={rgb, 255:red, 0; green, 0; blue, 0 }  ][line width=0.75]    (10.93,-3.29) .. controls (6.95,-1.4) and (3.31,-0.3) .. (0,0) .. controls (3.31,0.3) and (6.95,1.4) .. (10.93,3.29)   ;
			\draw [shift={(99,62)}, rotate = 1.21] [color={rgb, 255:red, 0; green, 0; blue, 0 }  ][line width=0.75]    (10.93,-3.29) .. controls (6.95,-1.4) and (3.31,-0.3) .. (0,0) .. controls (3.31,0.3) and (6.95,1.4) .. (10.93,3.29)   ;
			\draw    (131,107) .. controls (153.66,142.95) and (127.8,144.46) .. (130.84,181.77) ;
			\draw [shift={(131,183.5)}, rotate = 264.14] [color={rgb, 255:red, 0; green, 0; blue, 0 }  ][line width=0.75]    (10.93,-3.29) .. controls (6.95,-1.4) and (3.31,-0.3) .. (0,0) .. controls (3.31,0.3) and (6.95,1.4) .. (10.93,3.29)   ;
			\draw  [dash pattern={on 4.5pt off 4.5pt}]  (249,46.5) -- (250,125) ;
			\draw  [dash pattern={on 4.5pt off 4.5pt}]  (516,47.5) -- (517,126) ;
			\draw  [dash pattern={on 4.5pt off 4.5pt}]  (183,78.5) -- (234,78.5) ;
			\draw  [dash pattern={on 4.5pt off 4.5pt}]  (547,78.5) -- (598,78.5) ;
			\draw    (256,96.99) -- (443,96.51) ;
			\draw [shift={(445,96.5)}, rotate = 179.85] [color={rgb, 255:red, 0; green, 0; blue, 0 }  ][line width=0.75]    (10.93,-3.29) .. controls (6.95,-1.4) and (3.31,-0.3) .. (0,0) .. controls (3.31,0.3) and (6.95,1.4) .. (10.93,3.29)   ;
			\draw [shift={(254,97)}, rotate = 359.85] [color={rgb, 255:red, 0; green, 0; blue, 0 }  ][line width=0.75]    (10.93,-3.29) .. controls (6.95,-1.4) and (3.31,-0.3) .. (0,0) .. controls (3.31,0.3) and (6.95,1.4) .. (10.93,3.29)   ;
			\draw    (253,63) -- (517,63.5) ;
			\draw [shift={(519,63.5)}, rotate = 180.11] [color={rgb, 255:red, 0; green, 0; blue, 0 }  ][line width=0.75]    (10.93,-3.29) .. controls (6.95,-1.4) and (3.31,-0.3) .. (0,0) .. controls (3.31,0.3) and (6.95,1.4) .. (10.93,3.29)   ;
			\draw [shift={(251,63)}, rotate = 0.11] [color={rgb, 255:red, 0; green, 0; blue, 0 }  ][line width=0.75]    (10.93,-3.29) .. controls (6.95,-1.4) and (3.31,-0.3) .. (0,0) .. controls (3.31,0.3) and (6.95,1.4) .. (10.93,3.29)   ;
			\draw    (381.5,106.75) .. controls (404.16,142.7) and (378.3,144.21) .. (381.34,181.52) ;
			\draw [shift={(381.5,183.25)}, rotate = 264.14] [color={rgb, 255:red, 0; green, 0; blue, 0 }  ][line width=0.75]    (10.93,-3.29) .. controls (6.95,-1.4) and (3.31,-0.3) .. (0,0) .. controls (3.31,0.3) and (6.95,1.4) .. (10.93,3.29)   ;
			
			\draw (107,148) node [anchor=north west][inner sep=0.75pt]   [align=left] {$ $};
			\draw (110,77) node [anchor=north west][inner sep=0.75pt]   [align=left] {$\displaystyle K_{0}$};
			\draw (95,133) node [anchor=north west][inner sep=0.75pt]   [align=left] {0};
			\draw (645,133) node [anchor=north west][inner sep=0.75pt]   [align=left] {t};
			\draw (110,35) node [anchor=north west][inner sep=0.75pt]   [align=left] {$\displaystyle \tilde{K}_{0} +1$};
			\draw (140,131) node [anchor=north west][inner sep=0.75pt]   [align=left] {restart 1};
			\draw (61,189) node [anchor=north west][inner sep=0.75pt]   [align=left] {$\displaystyle x^{0,0} ,\ \dotsc ,\ x^{0,\tilde{K}_{0}} ,x^{1,0}$};
			\draw (221,130) node [anchor=north west][inner sep=0.75pt]   [align=left] {restart $\displaystyle \ell $};
			\draw (476,130) node [anchor=north west][inner sep=0.75pt]   [align=left] {restart $\displaystyle \ell +1$};
			\draw (330,75) node [anchor=north west][inner sep=0.75pt]   [align=left] {$\displaystyle K_{\ell }$};
			\draw (329,32) node [anchor=north west][inner sep=0.75pt]   [align=left] {$\displaystyle \tilde{K}_{\ell } +1$};
			\draw (305,189) node [anchor=north west][inner sep=0.75pt]   [align=left] {$\displaystyle x^{\ell ,0} ,\ \dotsc ,\ x^{\ell ,\tilde{K}_{\ell }} ,x^{\ell +1,0}$};
			
		\end{tikzpicture}
	} 
	\caption{Illustration of the restarting scheme and the corresponding notation.}
	\label{fig-restart}
\end{figure}

Using Lemma \ref{lemma-diff-f}, we obtain the following key lemma, which states that if no restart is performed, then the generated objective values are always less than or equal to the objective value of the initial point. 
\begin{lemma} \label{lemma-re-descend}
	Let $\{x^t\}$ be the sequence generated by Algorithm \ref{alg-apg}. Then, it holds that 
	\begin{equation}
		\label{eq-lemma-re-descend-1}
		f(x^{{t}}) \leq f(x^0),\quad \forall \; {t}\geq 0. 
	\end{equation}
	Moreover, it holds that 
	\[
	f(x^{\ell, k_\ell}) \leq f(x^{\ell, 0}),\quad \forall\; \ell \geq 0,\; 0\leq k_\ell \leq \widetilde{K}_\ell,
	\]
	and 
	\[
	f(x^{\ell+1, 0}) \leq f(x^{\ell, 0}),\quad \forall\; \ell\geq 0.      
	\]
\end{lemma}
\begin{proof}
	First, we assume that the restart does not occur until the $(t+1)$-th iteration. From Lemma \ref{lemma-diff-f} and using the fact that
    $$\inprod{\mathbf{b}-\mathbf{a}}{S(\mathbf{b}-\mathbf{a})} + 2\inprod{\mathbf{a}-\mathbf{c}}{S(\mathbf{b}-\mathbf{a})} = \inprod{\mathbf{b}-\mathbf{c}}{S(\mathbf{b}-\mathbf{c})} - \inprod{\mathbf{a}-\mathbf{c}}{S(\mathbf{a}-\mathbf{c})}$$
    for all $\mathbf{a}, \mathbf{b}, \mathbf{c}\in\mathbb{R}^n$, it is easy to check that 
	\[
	f(x^j) - f(x) \leq \frac{1}{2}\inprod{x - y^j}{S(x - y^j)} - \frac{1}{2}\inprod{x - x^j}{S(x - x^j)},\quad \forall \;0\leq j\leq t,
	\]
    {by letting $\mathbf{a}=y^t$, $\mathbf{b} = x^j$, and $\mathbf{c} = x$.} In particular, for $j = 1$ {and $x = x^0$}, we have 
	\begin{align*}
		f(x^1) - f(x^0) \leq &\; \frac{1}{2}\inprod{x^0 - y^1}{S(x^0 - y^1)} - \frac{1}{2}\inprod{x^0 - x^1}{S(x^0 - x^1)} \\
		= &\;  - \frac{1}{2}\inprod{x^0 - x^1}{S(x^0 - x^1)} \leq 0.
	\end{align*}
	For all $1\leq j\leq t-1$ {and $x = x^j$}, we also have 
	\begin{align*}
		&\; f(x^{j+1}) - f(x^j)\\ 
		\leq &\; \frac{1}{2}\inprod{x^j - y^{j+1}}{S(x^j - y^{j+1})} - \frac{1}{2}\inprod{x^j - x^{j+1}}{S(x^j - x^{j+1})} \\
		= &\; \frac{1}{2} \left(\frac{\alpha_{j}-1}{\alpha_{j+1}}\right)^2 \inprod{x^j - x^{j-1}}{S(x^{j} - x^{j-1})} - \frac{1}{2}\inprod{x^{j+1} - x^{j}}{S(x^{j+1} - x^{j})}.
	\end{align*}
	Summing {the above} over $j = 1, 2, \dots, t-1$, we obtain that
	\begin{align*}
		&\; f(x^t) - f(x^1)\\
		\leq &\; \frac{1}{2}\sum_{j = 2}^{t-1}\left(\left(\frac{\alpha_j-1}{\alpha_{j+1}}\right)^2-1\right)\inprod{x^j - x^{j-1}}{S(x^j - x^{j-1})} \\
		&\; + \frac{1}{2}\left(\frac{\alpha_1-1}{\alpha_2}\right)^2\inprod{x^1 - x^0}{S(x^1 - x^0)} - \frac{1}{2}\inprod{x^t - x^{t-1}}{S(x^t - x^{t-1})}.
	\end{align*}
	Recall that $\alpha_1 = 1$, {we see that the second term on the right-hand-side vanishes. Moreover, one can observe from the relation $\alpha_{t+1} = (1+\sqrt{1+4\alpha_t^2})/2$} that 
	\[
	\left(\frac{\alpha_j-1}{\alpha_{j+1}}\right)^2-1 = \frac{\alpha_j^2 - 2\alpha_j + 1 - \alpha_{j+1}^2}{\alpha_{j+1}^2} = \frac{2\sqrt{1+4\alpha_t^2}-8\alpha_t + 2}{4\alpha_{j+1}^2} < 0,
	\]
	where the last inequality is due to the fact that $\alpha_t \geq 1$ for $t\geq 1$. Then, we see that 
	\[
	f(x^t) - f(x^1) \leq -\frac{1}{2}\inprod{x^t - x^{t-1}}{S(x^t - x^{t-1})}.
	\]
    {This shows that the objective values are always not larger than the objective value at the initial piont if the restart is not performed.}
	
	Similarly, since the restart does not occur from $x^{\ell, 0}$ to $x^{\ell, \widetilde{K}_\ell}$, we can show that 
	\[
	f(x^{\ell, k_\ell}) \leq f(x^{\ell, 0}),\quad \forall\; 0\leq k_\ell \leq \widetilde{K}_\ell. 
	\]
    In particular, when $k_\ell = \widetilde{K}_\ell$, it holds that 
    \[
        f(x^{\ell+1, 0}) = f(x^{\ell, \widetilde{K}_{\ell}}) \leq f(x^{\ell, 0}),\quad \forall \ell\,\geq 0.
    \]
	Thus, the proof is completed.
\end{proof}

The above lemma indicates that the sequence $\{f(x^{\ell,0})\}$ is nonincreasing. Using this fact and assuming that $Q\in\S_{++}^n$, we are able to establish the convergence of Algorithm \ref{alg-apg-re}. In this case, one can see that the function $f$ is coercive and the level set 
\[
\mathcal{L}(x^0):=\{x\in \R^n\;:\; f(x)\leq f(x^0)\}
\]
is bounded. Consequently, there exist a constant $R > 0$ such that 
\[
R \geq \sup_{x\in \mathcal{L}(x^0)} \norm{x - x^*}_S.
\]
\begin{theorem}
	\label{thm-convergence-re}
	Let the sequence $\{x^t:=x^{\ell,k_\ell}\}$ be generated by Algorithm \ref{alg-apg-re}. Suppose that $Q\in\mathbb{S}_{++}^n$.
	Then, it holds that 
	\[
	f(x^t) - f(x^*)\leq \frac{2R^2(\log_2t)^2}{t^2},\quad \forall\; t\geq 2.
	\]
\end{theorem}
\begin{proof}
	From Theorem \ref{thm-convergence} and Lemma \ref{lemma-re-descend}, we have 
	\[
	f(x^{\ell, k_{\ell}}) - f(x^*) \leq \frac{2\norm{x^{\ell, 0} - x^*}_S^2}{(k_{\ell}+1)^2} \leq \frac{2R^2}{(k_{\ell}+1)^2}, \quad \forall \; \ell\geq 0,\; 0\leq k_\ell\leq \widetilde{K}_\ell,
	\]
	and for all $1 \leq j \leq \ell$,
	\begin{align*}
		f(x^{\ell, k_\ell}) - f(x^*) \leq &\;  f(x^{\ell, 0}) - f(x^*) 
		\leq f(x^{j, 0}) - f(x^*) \\
		= &\;  f(x^{j-1, \widetilde{K}_{j-1}}) - f(x^*) 
		\leq  \frac{2\norm{x^{j-1,0} - x^*}_S^2}{(\widetilde{K}_{j-1}+1)^2},
	\end{align*}
	which implies that $f(x^{\ell, k_\ell}) - f(x^*) \leq \frac{2R^2}{\widetilde{K}_{j-1}^2}$, for $1\leq j\leq \ell$. 
	
	Now we consider the iterate $x^{\ell, k_{\ell}}$, which corresponds to the point $x^t$ with $t = \sum_{j = 0}^{\ell-1} \widetilde{K}_j + k_{\ell} \geq 2$ and $0\leq k_{\ell} \leq \widetilde{K}_{\ell}$. Since $\widetilde{K}_\ell \geq K_\ell = K_0\cdot 2^{\ell}$ for all $\ell\geq 0$, we see that 
	\[ 
	f(x^t)-f(x^*) = f(x^{\ell, k_{\ell}}) - f(x^*) 
	\leq \frac{2R^2}{\max\left\{\max\left\{(\widetilde{K}_{j-1}+1)
		^2\right\}_{0\leq j\leq \ell-1}, (k_{\ell}+1)^2\right\}}.
	\]
	Moreover, since $t = \sum_{j=0}^{\ell-1}\widetilde{K}_j + k_\ell \geq \sum_{j=0}^{\ell-1}\widetilde{K}_j \geq \sum_{j=0}^{\ell-1}K_j = \frac{K_0(2^\ell-1)}{2-1}\geq 2^{\ell+1}-2$, we can check that the number of restarts $\ell$ satisfies that 
	\[
	\ell\leq \log_2(t+2) - 1. 
	\]
	As a consequence, we have that 
	\[
	\max\left\{\max\left\{(\widetilde{K}_{j-1}+1)
	^2\right\}_{0\leq j\leq \ell-1}, (k_{\ell}+1)^2\right\} \geq \frac{t}{\ell+1} \geq \frac{t}{\log_2(t+2)},
	\]
	which implies that 
	\[
	f(x^t)-f(x^*) \leq \frac{2R^2(\log_2t)^2}{t^2},\quad \forall\; t\geq 2.
	\]
	This completes the proof.
\end{proof}

We emphasize that the positive definiteness of $Q$ is necessary in our analysis since it ensures the existence of the positive constant $R$ which serves as an upper bound for the diameter of the level set $\mathcal{L}(x^0):=\{x\in \R^n\;:\; f(x)\leq f(x^0)\}$. In practice, we observe that Algorithm \ref{alg-apg-re} still works efficiently for linear systems with $Q\in \S_+^n$. {We should emphasize that the condition $\inprod{Qy^t - b}{x^t - x^{t-1}} \geq 0$ is not used in the convergence analysis. Instead, it is monitored in the algorithm as a heuristic to decide when to restart the extrapolation, which we found to improve practical efficiency. In contrast, the key to our convergence proof is the ability to avoid frequent restarts, which is ensured by enforcing the condition $t > K_{\rm} + K_\ell$.} Moreover, the convergence rate given by Theorem \ref{thm-convergence-re} introduces a $\log_2^2(t)$ term which is worse than the one given by Theorem \ref{thm-convergence}. Although Algorithm \ref{alg-apg-re} has a poorer 
worst-case convergence rate than that of Algorithm \ref{alg-apg}, it has better practical numerical performance empirically. It remains an open problem whether the restarting scheme admits better theoretical convergence properties. {Moreover, for a given problem, determining the optimal timing for restarting remains an open question that warrants further investigation. More advanced analytical tools and fully exploring the problem structure are required to bridge these gaps, which we leave as a direction for future research}. 

\section{A new proximal term for parallelization} \label{section-proximal-term}

In this section, we show how to design a new proximal term such that the resulting algorithm has a Jacobi-type updating rule. Since our main focus of this paper is to solve large-scale linear systems, it is desirable to simulate an environment of parallel/distributed platform. To this end, we shall partition the coefficient matrix $Q$ as follows:
\[
Q = \begin{pmatrix}
	Q^{1,1} & Q^{1,2} & \dots  & Q^{1,s} \\
	Q^{2,1} & Q^{2,2} & \dots  & Q^{2,s} \\
	\vdots  & \vdots  & \ddots & \vdots  \\
	Q^{s,1} & Q^{s,2} & \dots  & Q^{s,s}
\end{pmatrix} = L + D + L^T,
\]
with $Q^{i,j} = (Q^{j,i})^T\in \R^{n_i\times n_j}$, for $1\leq i,j\leq s$, $\sum_{i = 1}^sn_i = n$, and 
\[
L := \begin{pmatrix}
	0 & 0 & \dots & 0 \\
	Q^{2,1} & 0 & \dots & 0 \\
	\vdots & \vdots & \ddots & \vdots \\
	Q^{s,1} & Q^{2,s} & \dots & 0
\end{pmatrix},\quad D:= \diag\left(Q^{1,1},\dots,Q^{s,s}\right).
\]
According to the above partition, we also decompose any vectors $x \in \R^n$ as follows:
\[
x = \begin{pmatrix}
	x^{1} \\ \vdots \\ x^{s}
\end{pmatrix} \in \R^n,\quad x^{i} \in \R^{n_i},\; i = 1,\dots, s.
\]

Now, we consider choosing a matrix $J$ as a block-diagonal matrix, i.e., $J = \diag(J^{1,1}, \dots, J^{s,s})$ with $J^{i,i}$ being a symmetric positive definite matrix. Let $S:= J - Q$. Then, to get $x^t$ in Line 3 of Algorithm~\ref{alg-apg-re}, it suffices to solve the following linear system
\[
x^t = y^t + J^{-1}(b - Qy^t).
\]
which is indeed a Jacobi-type updating rule and can readily be implemented in a parallel manner. Assuming that $S$ is positive semidefinite, the template of the parallel version using the above matrix $S$ with an adaptive restarting strategy can be given in Algorithm \ref{alg-jacobi}. 

\begin{algorithm}[htb!]
	\begin{algorithmic}[1]
		\STATE \textbf{Input:} Initial points $x^{0} = y^1 \in \R^n$, $\alpha_0 = 0,\;\alpha_1 = 1$, {the matrix $J:=\diag(J_{1}, \dots, J_{s})\in \S_{++}^n$ defined by eq.(\ref{eq-J-diag}) or (\ref{eq-J-blkdiag})}, $\textrm{IsRestart}\in \{\textrm{True},\textrm{False}\}$, $\ell = 0$, $K_{\ell}\geq 2$ and $K_{\textrm{re}} = 0$.
		\FOR{$t\geq 1$}
		\FOR{$i = 1,2,\dots,s$ \textbf{parallel}}
		\STATE $(x^t)^i = (y^t)^i + (J_{i})^{-1}\left(b^i - \sum_{j = 1}^sQ^{i,j}(y^t)^j\right)$.
		\ENDFOR
		\IF{$t > K_{\textrm{re}} + K_{\ell}$ and $\inprod{Qy^t - b}{x^t - x^{t-1}} \geq  0$ and IsRestart}
		\STATE $K_{\textrm{re}} = t$.
		\STATE $K_{\ell+1} = 2K_{\ell}$.
		\STATE $\ell = \ell + 1$.
		\STATE $\alpha_{t+1} = 1$, $\alpha_t = 0$.
		\STATE $y^{t+1} = x^{t-1}$, $x^t = x^{t-1}$.
		\ELSE
		\STATE $\alpha_{t+1} = \frac{1+\sqrt{1+4\alpha_t^2}}{2}$.
		\STATE $y^{t+1} = \left(1+\frac{\alpha_t - 1}{\alpha_{t+1}}\right)x^{t} - \frac{\alpha_t - 1}{\alpha_{t+1}}x^{t-1}$.
		\ENDIF 
		\ENDFOR 
		\STATE \textbf{Output:} $x^t$
	\end{algorithmic}
	\caption{A parallel Nesterov's accelerated Jacobi-type method with an adaptive restarting strategy.}
	\label{alg-jacobi}
\end{algorithm}

To apply the convergence result given by Theorem \ref{thm-convergence}, we require that $S\in \S_+^n$. And to apply the convergence result given by Theorem \ref{thm-convergence-re}, we require that $Q\in \S_{++}^n$ and $S\in \S_+^n$. We note that the 
semi-definiteness of the matrix $S$ can be easily satisfied. For example, if we want $J$ to be a diagonal matrix, we may set its diagonal entries as 
\begin{equation}
	\label{eq-J-diag}
	J_{kk}:= Q_{kk} + \sum_{j\neq k}\left\lvert Q_{kj}\right\rvert ,\quad k = 1,\dots,n.
\end{equation}
On the other hand, we can define
\begin{equation}
	\label{eq-J-blkdiag}
	J^{i,i} = Q^{i,i} + \sum_{j \neq i}\norm{Q^{i,j}}_2I_{n_i},\quad i = 1,\dots, s,
\end{equation}
to get a block diagonal $J$. For both cases, one can check that if $Q$ contains no zero row/column, then $S\in \mathbb{S}_+^n$ and $J\in\mathbb{S}^{n}_{++}$, respectively.

\begin{remark}\label{remark-preconditioner}
    While our framework does not explicitly incorporate the same preconditioning techniques used in the conjugate gradient method, we argue that the matrix $J$ can be interpreted as a preconditioner. In this work, we primarily focus on block-diagonal matrices for $J$, as they enable efficient parallel implementation. However, more advanced strategies can be developed, provided that $J$ is invertible, and the matrix $S=J-Q$ remains positive semidefinite. The design of $J$ is highly problem-specific and depends on the underlying structure and requirements of the application. As a consequence, exploring potential problem structures within a specific application domain can be a more effective approach.
\end{remark}

\section{Numerical experiments} \label{section-experiments}

In this section, we conduct numerical experiments to verify the efficiency and scalability of the proposed algorithms. The outline and main goals of our numerical experiments are summarized as follows:
\begin{enumerate}
	\item We first construct a set of simple strictly diagonal dominant systems for which the classical Jacobi and weighted Jacobi methods are both guaranteed to be convergent. The purpose of this set of numerical tests is to show that our proposed algorithms outperform the classical Jacobi-type methods. 
	\item Next, in order to evaluate the performance on SPSD systems, we consider the Laplacian systems with underlying graph obtain from real applications. We compare the proposed algorithm against the CG method and the preconditioned CG method with a simple diagonal preconditioner. Since the (preconditioned) CG method is recognized as one of the most efficient iterative methods, our goal is to show that the proposed method has competitive practical performance. 
	\item We repeat the above tests on some {symmetric positive definite} systems with real data sets in order to evaluate the performance of the proposed algorithm on {symmetric positive definite} systems.
	\item Finally, we implement our algorithm on a parallel/distributed platform to evaluate the scalability of the algorithm that is beneficial from parallel and distributed computing. 
\end{enumerate}

In our experiment, we only test Algorithm \ref{alg-jacobi} with $J$ chosen as in \eqref{eq-J-diag}, since, based on our numerical experience, Algorithm \ref{alg-jacobi} with $J$ chosen as in \eqref{eq-J-blkdiag} shares a similar performance as the former one. We terminate all the tested methods based on the value of the relative {objective gap, i.e., $\frac{|f(x^t) - f^*|}{1 + |f^*|}$}, where $x^t$ is the generated iterate at $t$-th iteration of an algorithm. In particular, we set the termination tolerance $\texttt{tol}$ to {$10^{-10}$}. We also set the maximum number of iterations allowed, $\texttt{maxiter}$, to be $5000$.

The update rules for each algorithm are summarized in Table \ref{tab-complexity}. For the accelerated Jacobi method, we choose $J$ to be a diagonal matrix. As demonstrated, the differences in per-iteration computational costs are negligible, making a comparison based solely on the number of iterations sufficient. Additionally, the practical performance heavily depends on the implementation of matrix-vector multiplications. Leveraging state-of-the-art implementations can further optimize performance, but this aspect lies beyond the scope of this paper. In view of the above, we have excluded a comparison of wall-clock time. 

\begin{table}[h]
        \centering
        \begin{tabular}{c|c|c}\hline 
            Method & Update ($t\geq 0$) & Complexity\\ \hline 
            Jacobi & $x^{t+1} = x^t + D^{-1}(b - Qx^t)$ & $O(n^2)$ \\ \hline 
            Weighted Jacobi & $x^{t+1} = x^t + \omega D^{-1}(b - Qx^t)$ & $O(n^2)$ \\ \hline 
            \multirow{2}{*}{Accelerated Jacobi} & $x^{t+1} = y^t + J^{-1}(b - Qx^t)$ & \multirow{2}{*}{$O(n^2)$} \\
            & $y^{t+1} = \left(1 + \frac{\alpha_t-1}{\alpha_{t+1}}\right)x^t - \frac{\alpha_t-1}{\alpha_{t+1}}x^{t-1}$ &  \\ \hline 
            \multirow{3}{*}{Conjugate Gradient} &$ x^{t+1} = x^t + \frac{\langle r^t, r^t\rangle}{\langle p^t, Qp^t\rangle}p^t$ & \multirow{3}{*}{$O(n^2)$} \\  
            & $r^{t+1} = r^t - \frac{\langle r^t, r^t\rangle}{\langle p^t, Qp^t\rangle} Qp^t$ & \\ 
            & $p^{t+1} = r^{t+1} + \frac{\langle r^{t+1}, r^{t+1}\rangle}{\langle r^t, r^t\rangle}p^t$ & \\ \hline
        \end{tabular}
        \caption{Testing methods and their per-iteration complexities.}
        \label{tab-complexity}
    \end{table}

All the numerical experiments in Section \ref{subsection-sdd}, Section \ref{subsection-laplacian}, and Section \ref{subsection-spd} are conducted on a laptop with 2.2 GHz Quad-Core Intel Core i7 processor with 16 GB 1600 MHz DDR3 Memory, and the speed-up tests conducted in Section \ref{subsection-speedup} are performed on the Zaratan cluster operated by the University of Maryland. The cluster has 360 compute nodes, each with dual 2.45 GHz AMD 7763 64-core CPUs and HDR-100 (100 Gbit) Infiniband interconnects between the nodes, with storage and service nodes connected with full HDR (200 Gbit).

\subsection{Strictly diagonal dominant systems} \label{subsection-sdd}
Similar to \cite{zhang2009cuda}, we consider the following strictly diagonal dominant symmetric system:
\[
Q:= \begin{pmatrix}
	n & -1 & -1 & \cdots & -1 \\ 
	-1 & n & -1 & \cdots & -1 \\
	\vdots & \vdots & \vdots & \ddots & \vdots \\
	-1 & -1 & -1 & \cdots & n
\end{pmatrix}\in \S^n,\quad b = \begin{pmatrix}
	1 \\ 1 \\ \vdots \\ 1
\end{pmatrix}\in \R^n,
\]
where $n$ is the dimension of $Q$. It is known that both the classical Jacobi method and weighted Jacobi method converge when solving this linear system. For the weighted Jacobi method, we compute the optimal weight $\omega_{\textrm{opt}} = \frac{2}{\lambda_{\textrm{min}}(D^{-1}Q) + \lambda_{\textrm{max}}(D^{-1}Q)}$ in order to obtain the optimal performance. We test six systems with matrix orders $n\in \{1000, 2000, 3000, 4000, 5000, 6000\}$. The computational results are presented in Figure \ref{fig-sdd}.

\begin{figure}[htb!]
	\centering
	\begin{subfigure}[b]{0.3\textwidth}
		\centering
		\includegraphics[width=\textwidth]{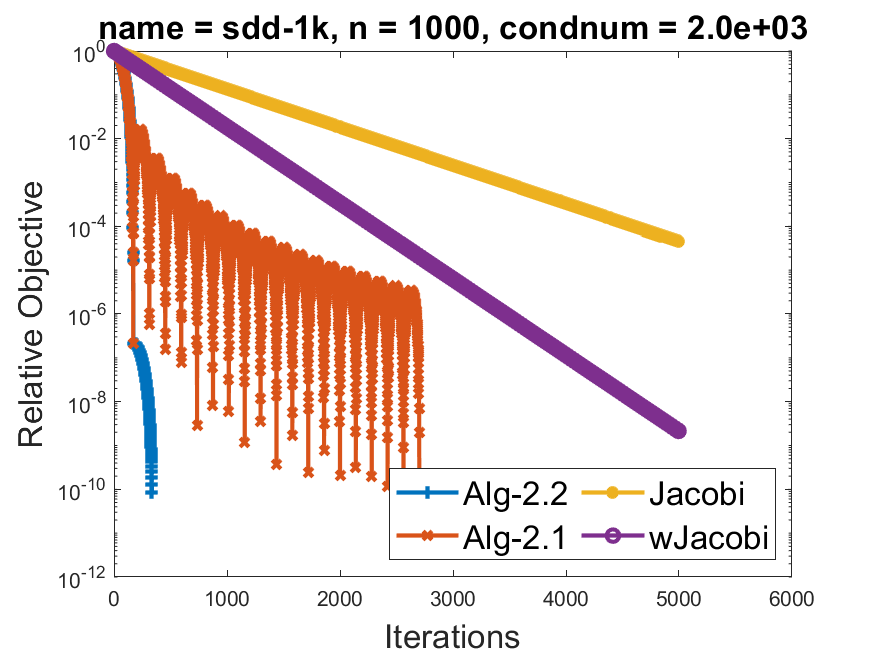}
		\subcaption{$n = 1000$}
	\end{subfigure}
	\hfill
	\begin{subfigure}[b]{0.3\textwidth}
		\centering
		\includegraphics[width=\textwidth]{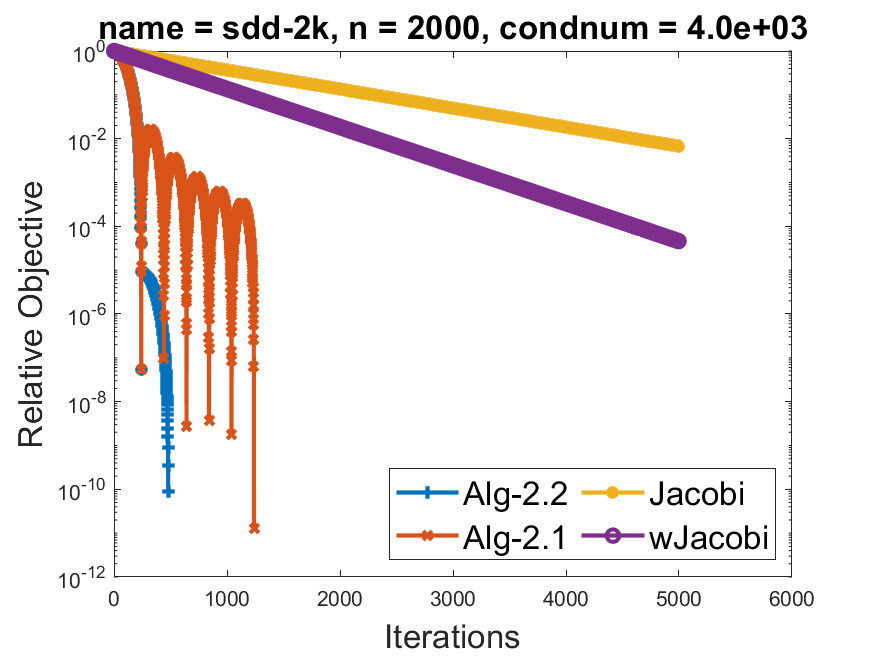}
		\subcaption{$n = 2000$}
	\end{subfigure}
	\hfill
	\begin{subfigure}[b]{0.3\textwidth}
		\centering
		\includegraphics[width=\textwidth]{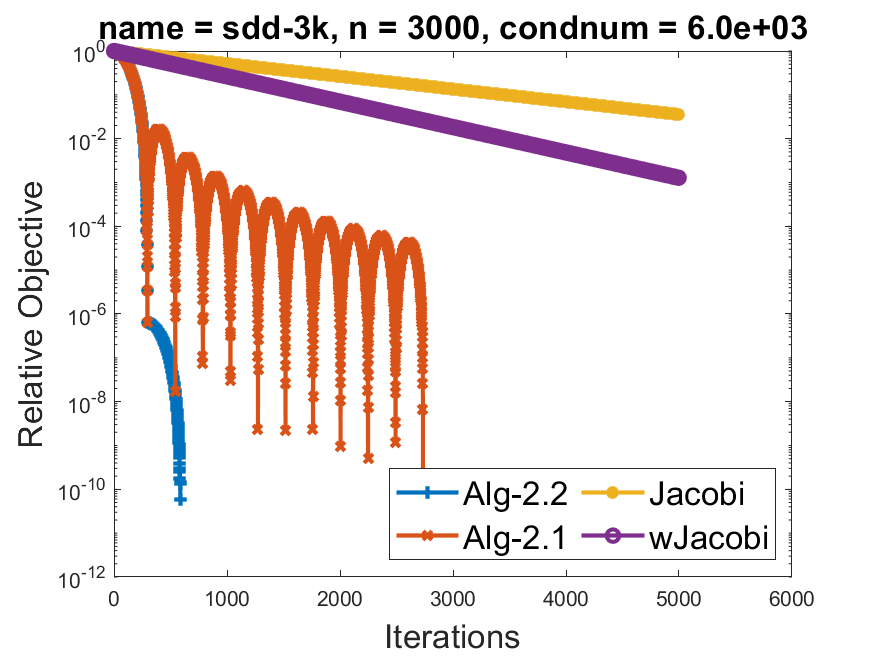}
		\subcaption{$n = 3000$}
	\end{subfigure}
	\hfill \\
	\begin{subfigure}[b]{0.3\textwidth}
		\centering
		\includegraphics[width=\textwidth]{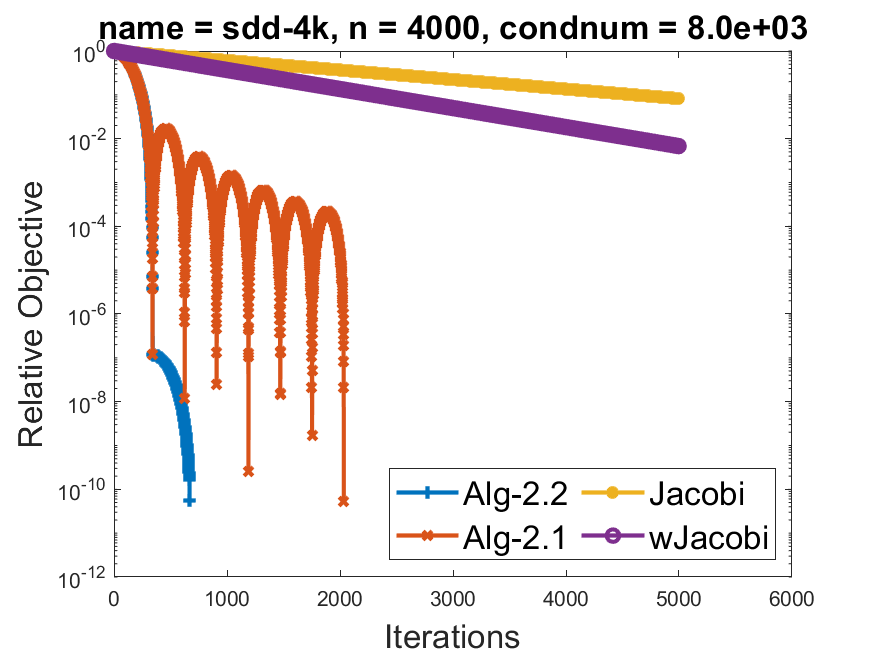}
		\subcaption{$n = 4000$}
	\end{subfigure}
	\hfill 
	\begin{subfigure}[b]{0.3\textwidth}
		\centering
		\includegraphics[width=\textwidth]{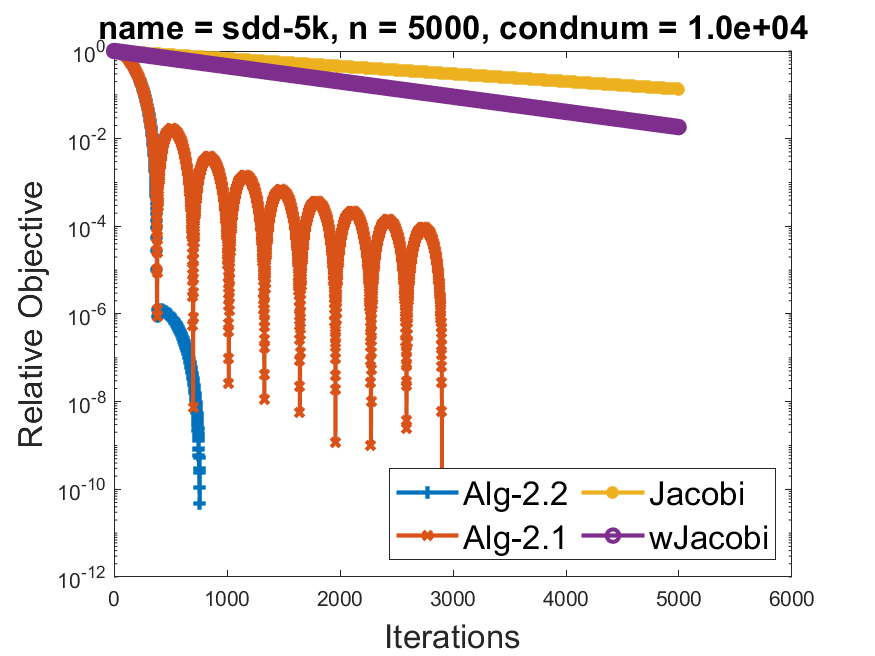}
		\subcaption{$n = 5000$}
	\end{subfigure}
	\hfill
	\begin{subfigure}[b]{0.3\textwidth}
		\centering
		\includegraphics[width=\textwidth]{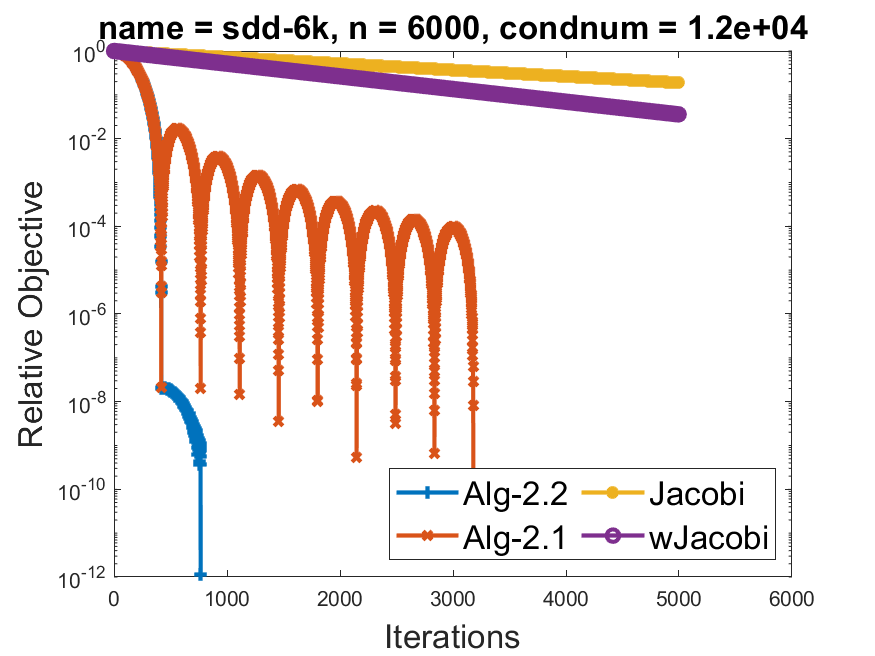}
		\subcaption{$n = 6000$}
	\end{subfigure}
	\caption{Computational results for strictly diagonally dominant systems.}
	\label{fig-sdd}
\end{figure}

From Figure \ref{fig-sdd}, we observe that both the Jacobi method and the weighted Jacobi method are convergent and admit linear convergence rates. Moreover, the weighted Jacobi method with the optimal weight has a faster convergence rate than that of the classical Jacobi method. However, both methods are not able to solve most problems to the desired accuracy {within 5000 iterations}. On the other hand, the proposed algorithms, {i.e., Algorithm \ref{alg-apg} and Algorithm \ref{alg-apg-re}} significantly outperform the other two methods and are able to solve all the problems successfully to the desired accuracy. {Furthermore, by leveraging the adapative restarting scheme, the overall performance is improved significantly.} These provide strong evidence that the Nesterov's accelerated technique equipped with the adaptive restarting starting strategy is indeed efficient. 

\subsection{Laplacian systems} \label{subsection-laplacian} 

Let $G=(V,E)$ be a simple graph with the set of nodes $V:=\{v_1,\dots, v_n\}$ and the set of edges $E=\{(i,j)\;:\; \textrm{$v_i$ is adjacent to $v_j$}, \;i\neq j\}$. The Laplacian matrix $L(G)\in \R^{n\times n}$ of $G$ is defined as 
\[
[L(G)]_{ij}:= \left\{
\begin{array}{ll}
	\textrm{deg}(v_i), & \textrm{if $i = j$}, \\
	-1, & \textrm{if $i\neq j$ and $(i,j)\in E$}, \\
	0, & \textrm{otherwise,}
\end{array}
\right. \quad 1\leq i, j\leq n,
\]
where $\textrm{deg}(v_i)$ denotes the degree of the node $v_i$, for $1\leq i \leq n$. It is well-known that the Laplacian matrix $L(G)$ is symmetric positive semidefinite and singular with $\textrm{rank}(L(G)) = n-1$. Solving Laplacian linear systems is 
a fundamental computational task with fruitful applications in a wide range of fields. 
The classical Jacobi method and the weighted Jacobi method are not guaranteed to be convergent when solving Laplacian systems. In fact, our numerical tests show that both methods are divergent empirically. Therefore, we only compare the proposed {Algorithm \ref{alg-apg-re}} with the CG method and the PCG method with a simple diagonal preconditioner. We select a set of graphs arising from real-world applications in the \texttt{SuiteSparse Matrix Collection}~\footnote{Available at \url{https://sparse.tamu.edu/}.} with various sizes. The corresponding numerical results are presented in Figure \ref{fig-laplacian}.

\begin{figure}[htb!]
	\centering
	\begin{subfigure}[b]{0.3\textwidth}
		\centering
		\includegraphics[width=\textwidth]{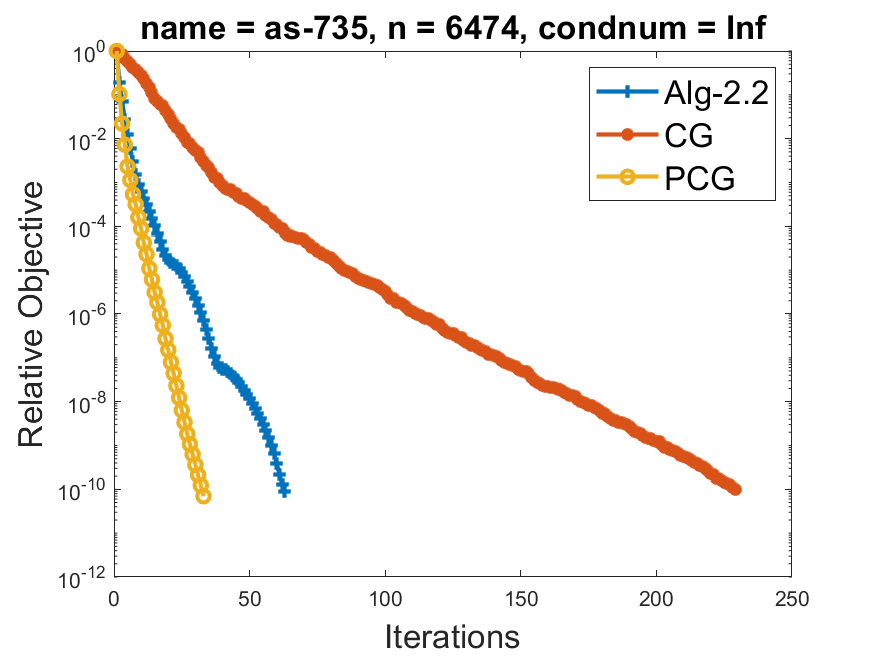}
		\subcaption{as-735}
	\end{subfigure}
	\hfill
	\begin{subfigure}[b]{0.3\textwidth}
		\centering
		\includegraphics[width=\textwidth]{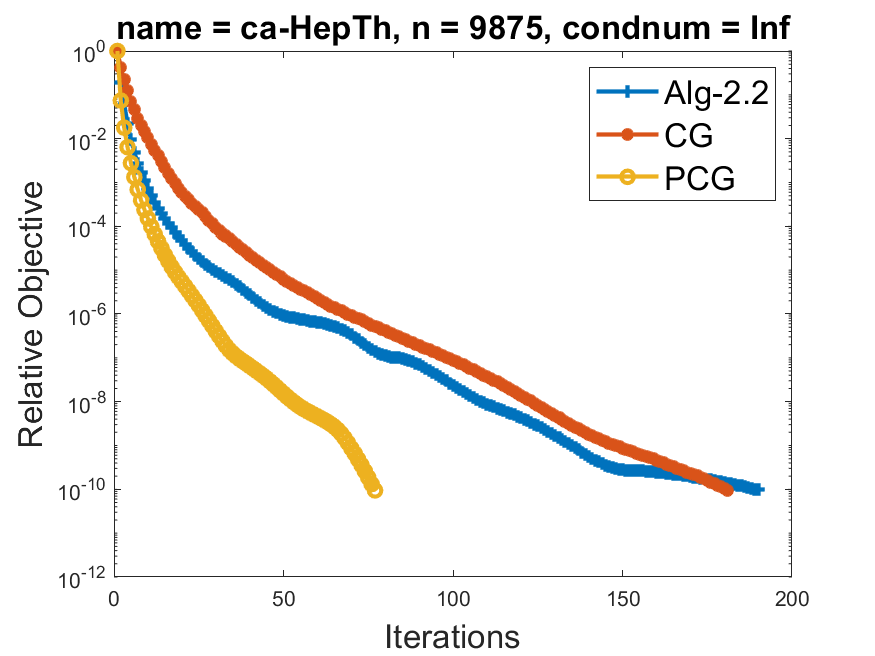}
		\subcaption{ca-HepTh}
	\end{subfigure}
	\hfill
	\begin{subfigure}[b]{0.3\textwidth}
		\centering
		\includegraphics[width=\textwidth]{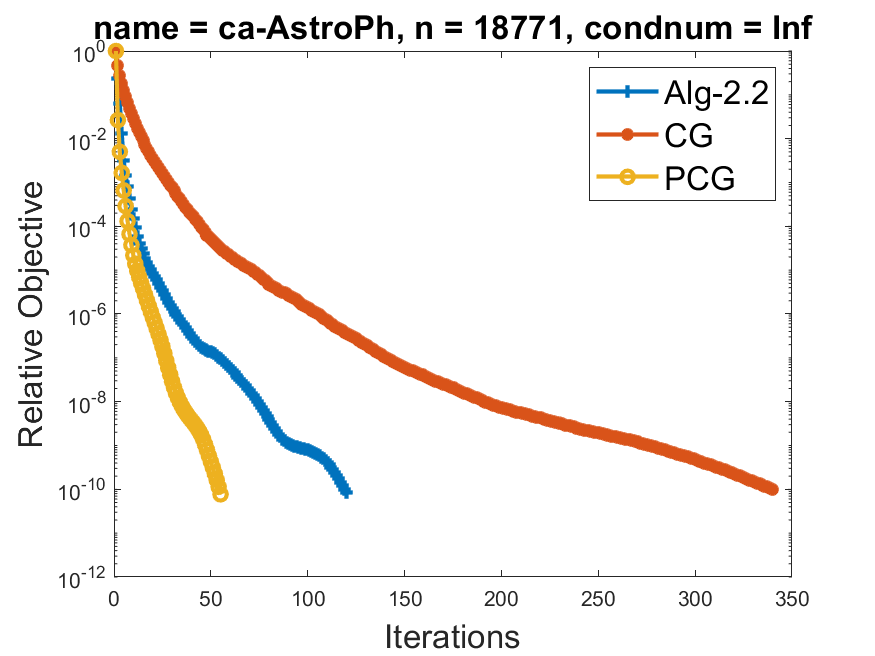}
		\subcaption{ca-AstroPh}
	\end{subfigure}
	\hfill \\
	\begin{subfigure}[b]{0.3\textwidth}
		\centering
		\includegraphics[width=\textwidth]{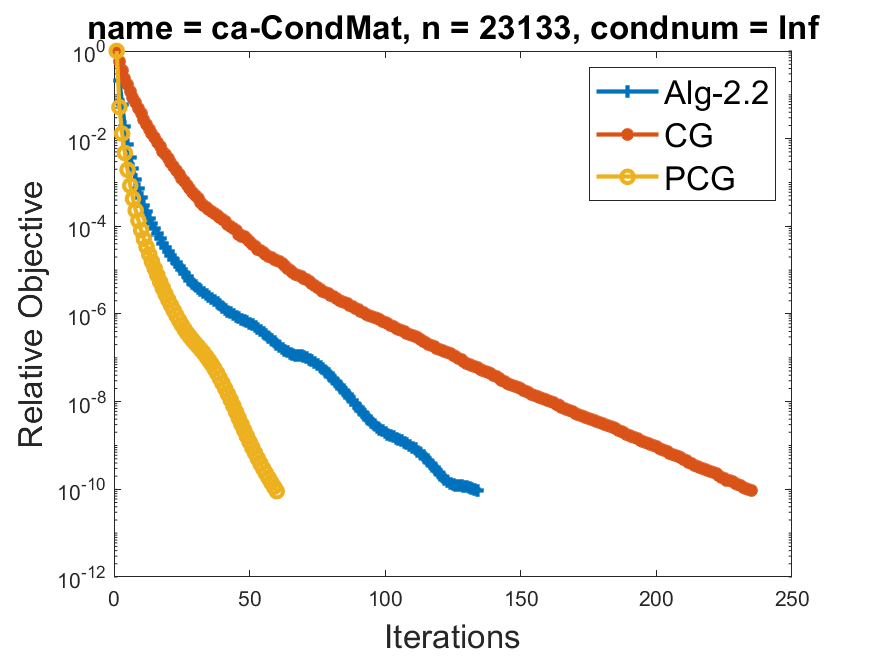}
		\subcaption{ca-CondMat}
	\end{subfigure}
	\hfill 
	\begin{subfigure}[b]{0.3\textwidth}
		\centering
		\includegraphics[width=\textwidth]{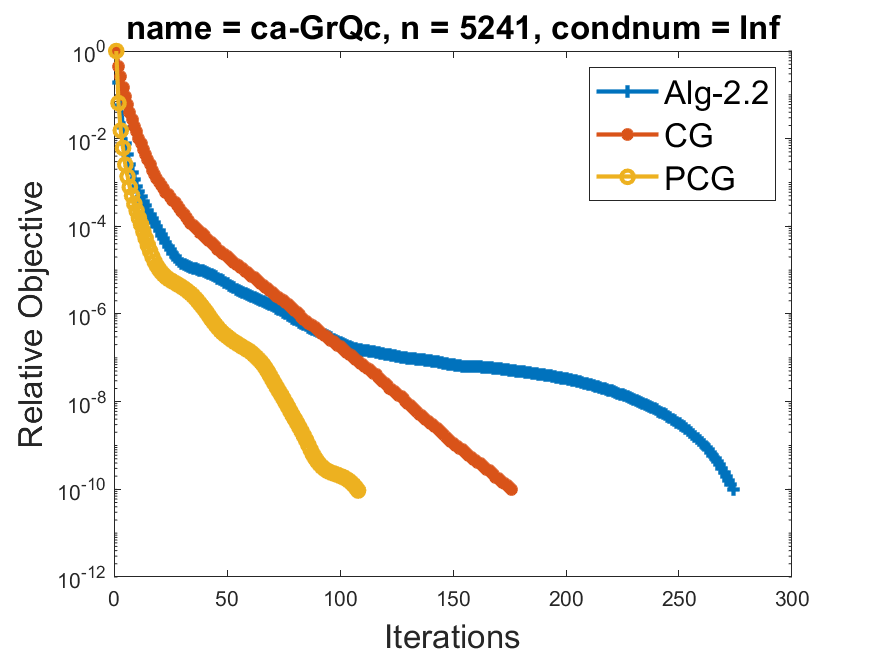}
		\subcaption{ca-GrQc}
	\end{subfigure}
	\hfill
	\begin{subfigure}[b]{0.3\textwidth}
		\centering
		\includegraphics[width=\textwidth]{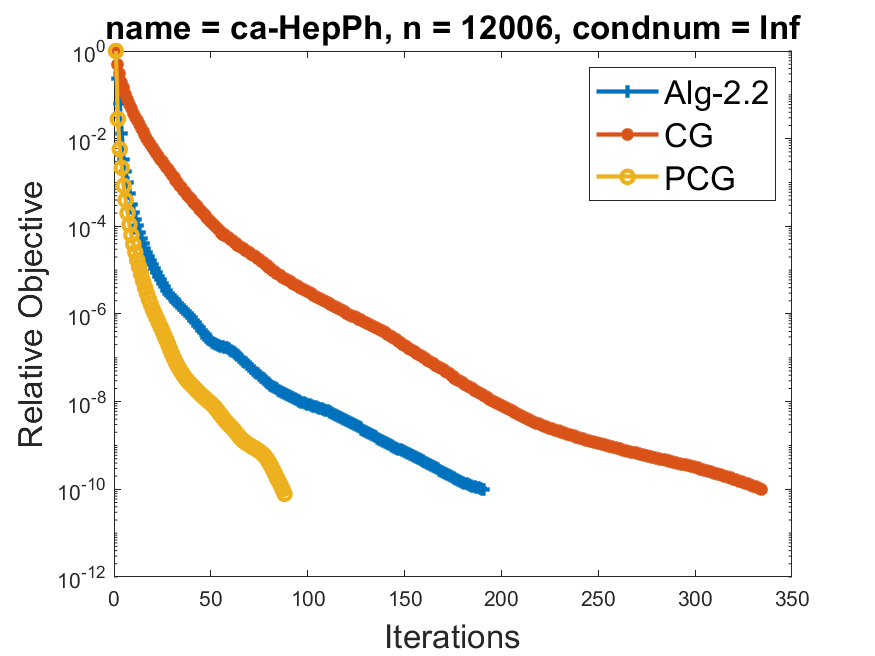}
		\subcaption{ca-HepPh}
	\end{subfigure}
	\hfill \\
	\begin{subfigure}[b]{0.3\textwidth}
		\centering
		\includegraphics[width=\textwidth]{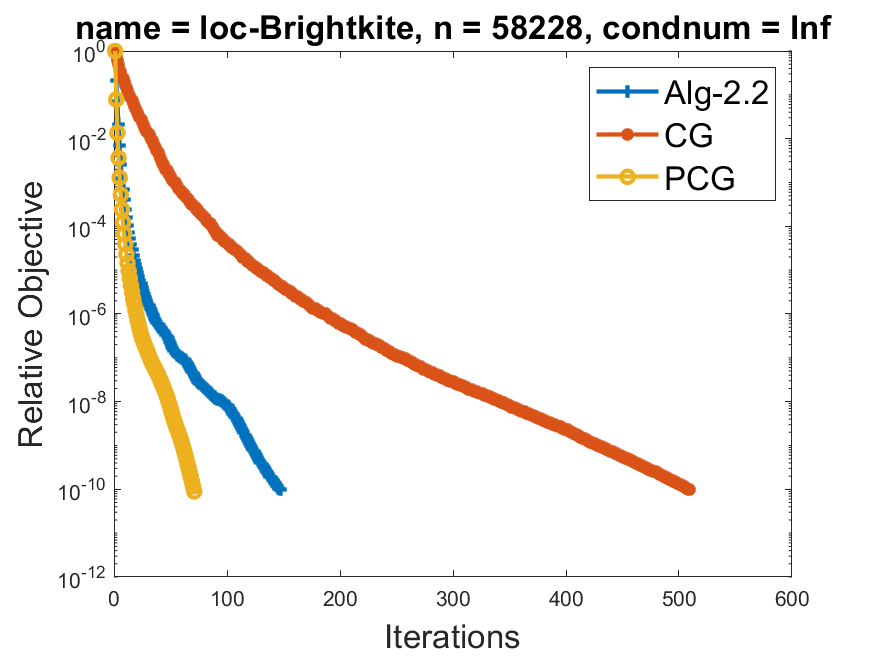}
		\subcaption{loc-Brightkite}
	\end{subfigure}
	\hfill
	\begin{subfigure}[b]{0.3\textwidth}
		\centering
		\includegraphics[width=\textwidth]{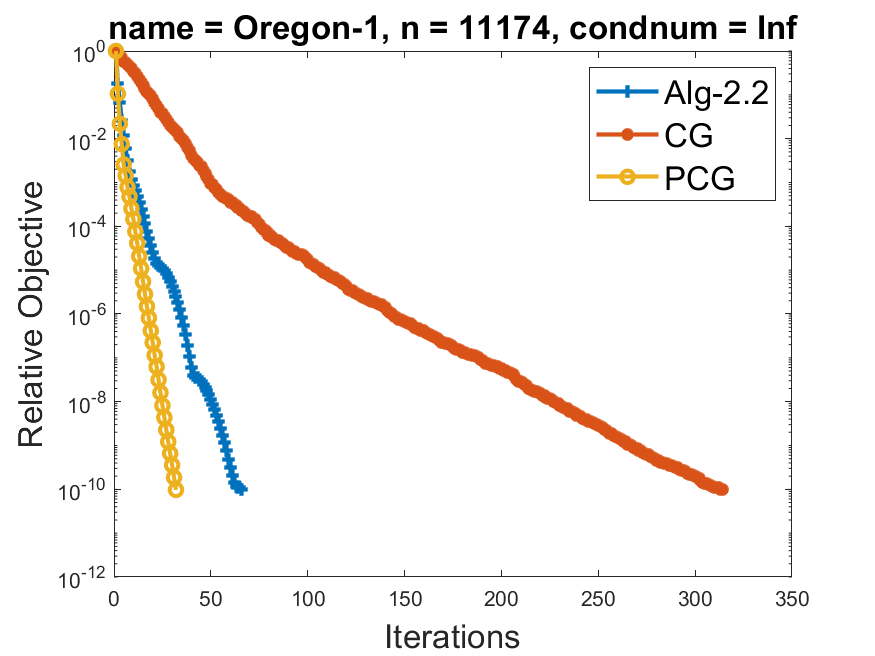}
		\subcaption{Oregon-1}
	\end{subfigure}
	\hfill
	\begin{subfigure}[b]{0.3\textwidth}
		\centering
		\includegraphics[width=\textwidth]{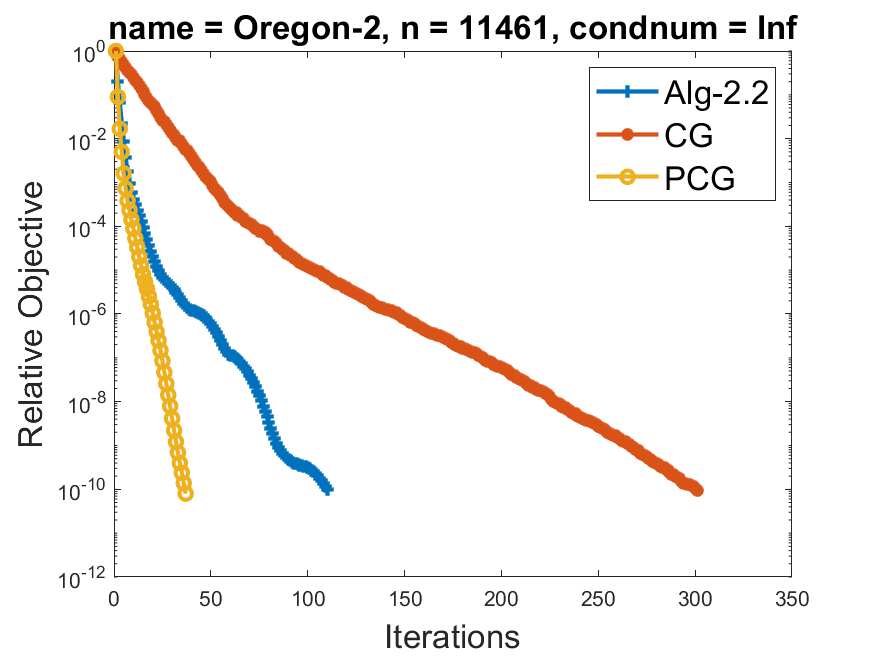}
		\subcaption{Oregon-2}
	\end{subfigure}
	\caption{Computational results for Laplacian systems.}
	\label{fig-laplacian}
\end{figure}

We can observe from Figure \ref{fig-laplacian} that all the three methods are quite efficient and are able to solve all the tested problems to the specified accuracy. We can also see that {Algorithm \ref{alg-apg-re}} and \texttt{PCG} have shown superior performance when compared to the pure CG method, i.e., \texttt{CG}. However, Algorithm \ref{alg-apg-re} performs slightly worse than the \texttt{PCG} method. In particular, the number of iterations taken by Algorithm \ref{alg-apg-re} is sometimes more than twice the number of iterations taken by \texttt{PCG}.  

\subsection{{Symmetric positive definite} systems} \label{subsection-spd}

We also test Algorithm \ref{alg-apg-re}, \texttt{CG} and \texttt{PCG} on {symmetric positive definite} systems whose coefficient matrices are also collected from the SuiteSparse Matrix Collection. These matrices are real-world data coming from different practical science and engineering problems. The numerical results presented in Figure \ref{fig-suitesparse} show that Algorithm \ref{alg-apg-re} {performs worse than} \texttt{PCG} and generally performs better than \texttt{CG}. 

{To assess the impact of the preconditioner, we also tested PCG equipped with the same diagonal preconditioner~$J$ used in our proposed method, where $J$ is defined as the vector of row (or column) absolute sums of the matrix~$Q$. Based on our observation, the resulting variant exhibits performance nearly identical to that of the vanilla CG method. For simplicity, the detailed computational results are not included in the paper. This suggests that the observed performance differences are not primarily due to the choice of preconditioner. While using the diagonal of~$Q$ as~$J$ may seem natural, it can violate the condition required by our analysis. However, problem-specific choices of~$J$ may offer further improvements in efficiency.
}

\begin{figure}[htb!]
	\centering
	\begin{subfigure}[b]{0.3\textwidth}
		\centering
		\includegraphics[width=\textwidth]{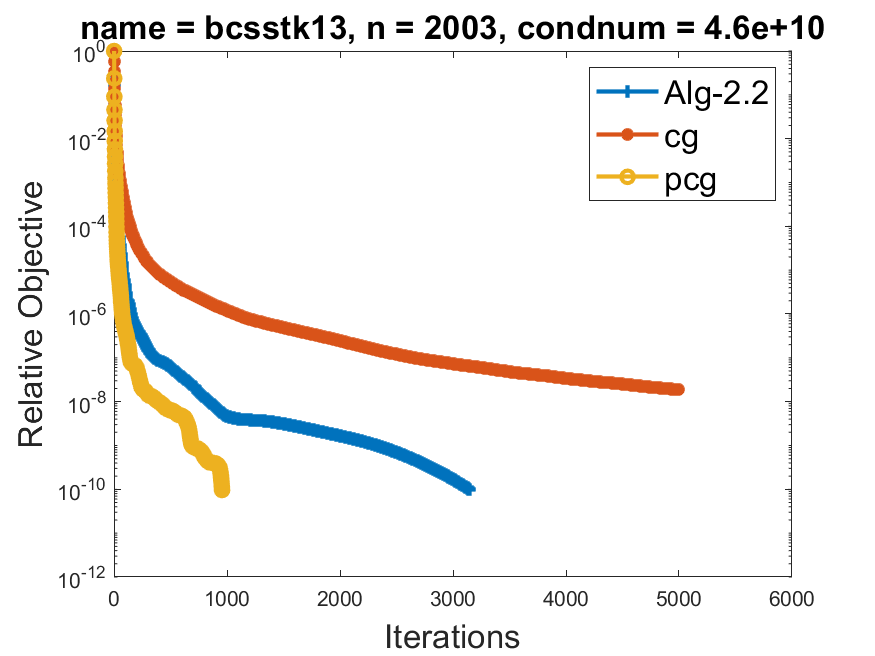}
		\subcaption{bcsstk13}
	\end{subfigure}
	\hfill 
	\begin{subfigure}[b]{0.3\textwidth}
		\centering
		\includegraphics[width=\textwidth]{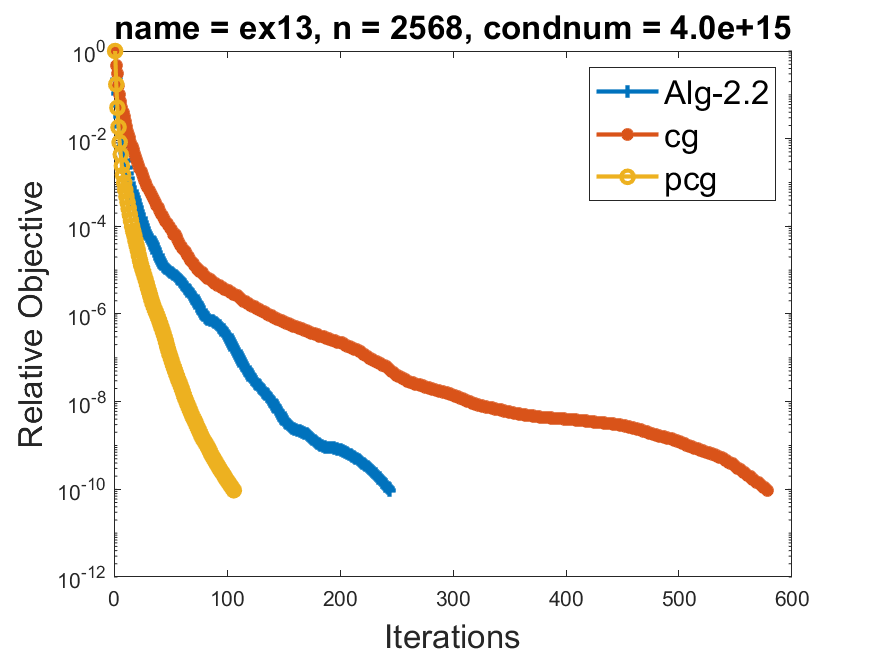}
		\subcaption{ex13}
	\end{subfigure}
	\hfill
	\begin{subfigure}[b]{0.3\textwidth}
		\centering
		\includegraphics[width=\textwidth]{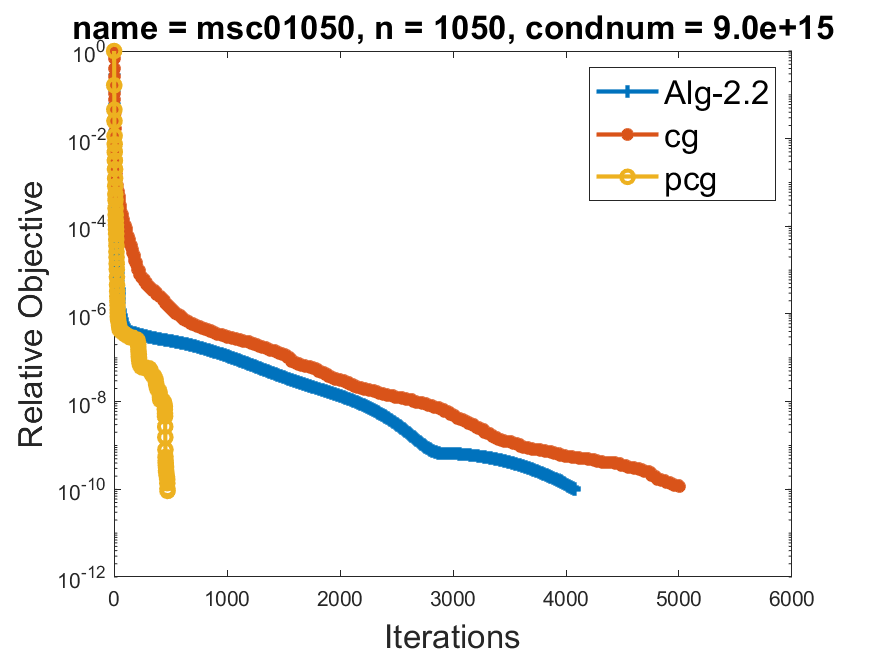}
		\subcaption{msc01050}
	\end{subfigure}
	\hfill \\
	\begin{subfigure}[b]{0.3\textwidth}
		\centering
		\includegraphics[width=\textwidth]{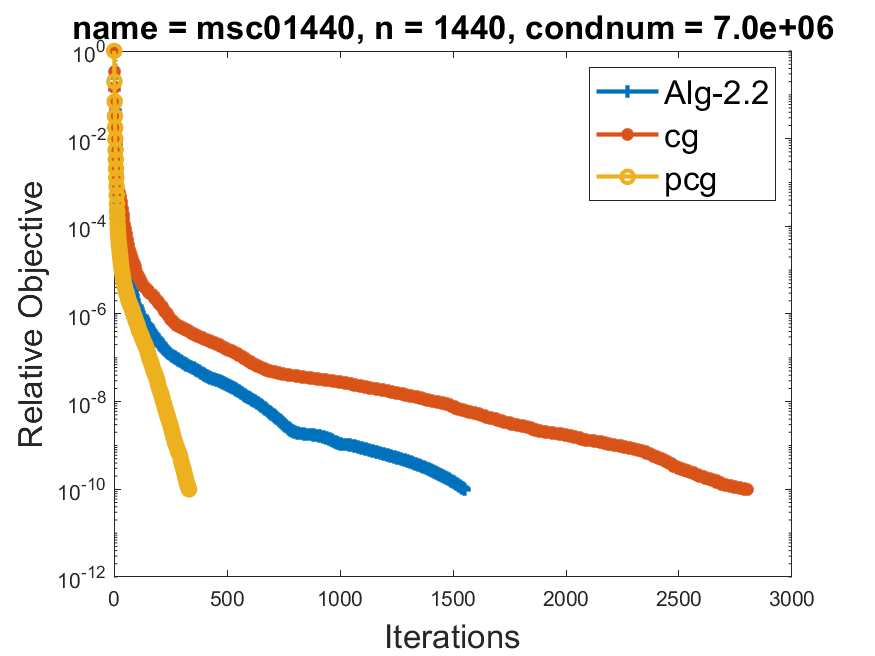}
		\subcaption{msc01440}
	\end{subfigure}
	\hfill
	\begin{subfigure}[b]{0.3\textwidth}
		\centering
		\includegraphics[width=\textwidth]{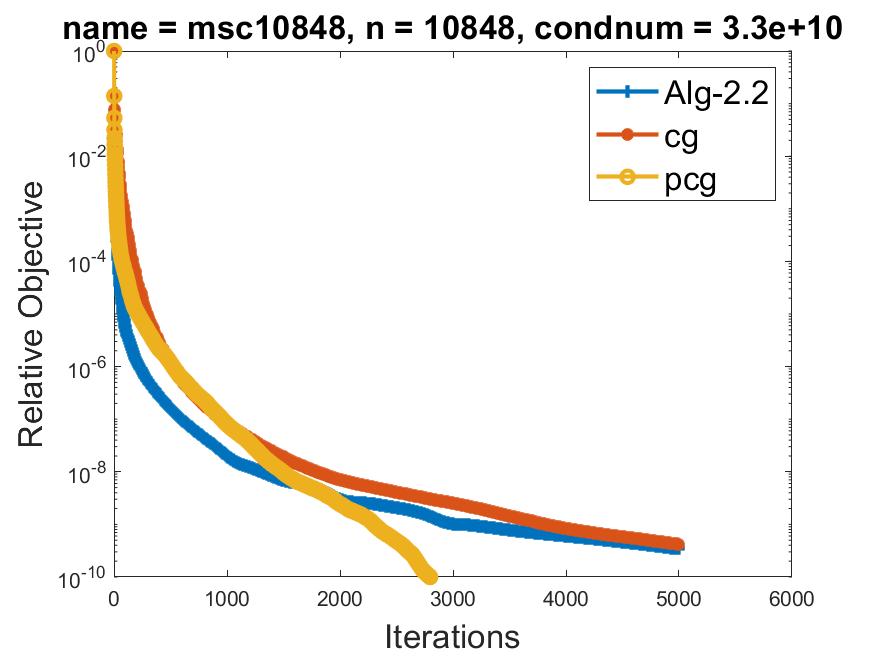}
		\subcaption{msc10848}
	\end{subfigure}
	\hfill
	\begin{subfigure}[b]{0.3\textwidth}
		\centering
		\includegraphics[width=\textwidth]{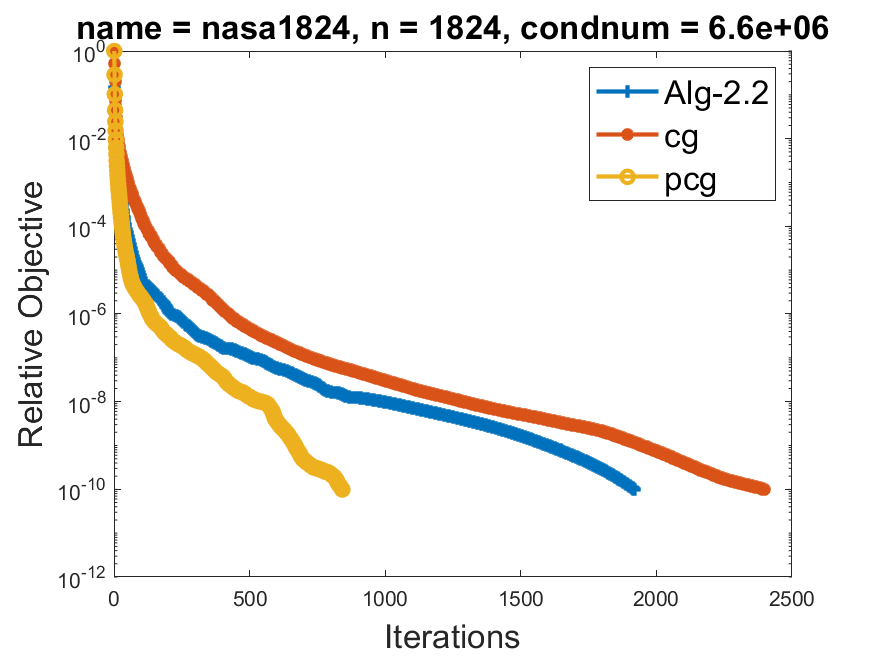}
		\subcaption{nasa1824}
	\end{subfigure} 
    \hfill \\
    \begin{subfigure}[b]{0.3\textwidth}
		\centering
		\includegraphics[width=\textwidth]{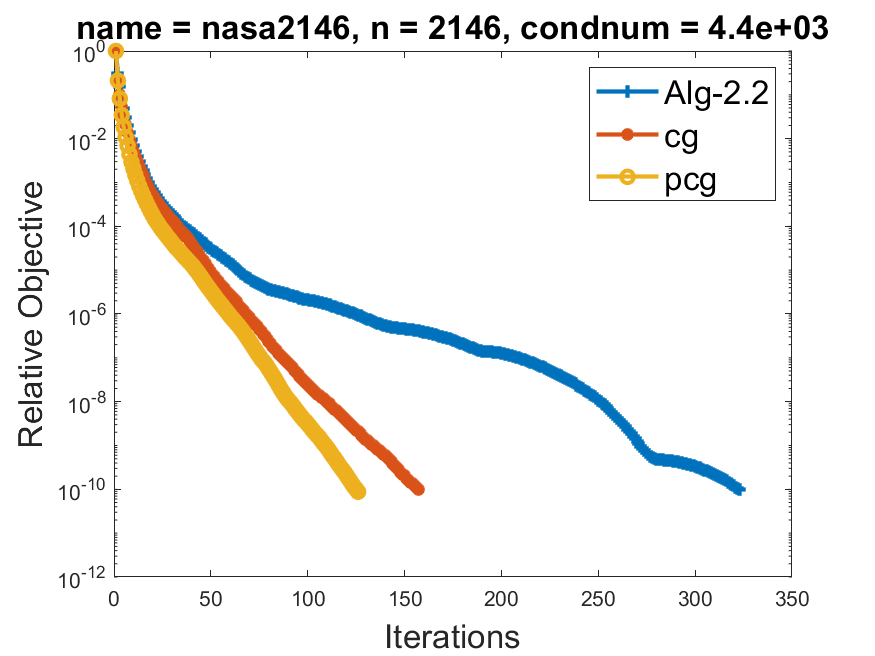}
		\subcaption{nasa2146}
	\end{subfigure}
	\hfill
	\begin{subfigure}[b]{0.3\textwidth}
		\centering
		\includegraphics[width=\textwidth]{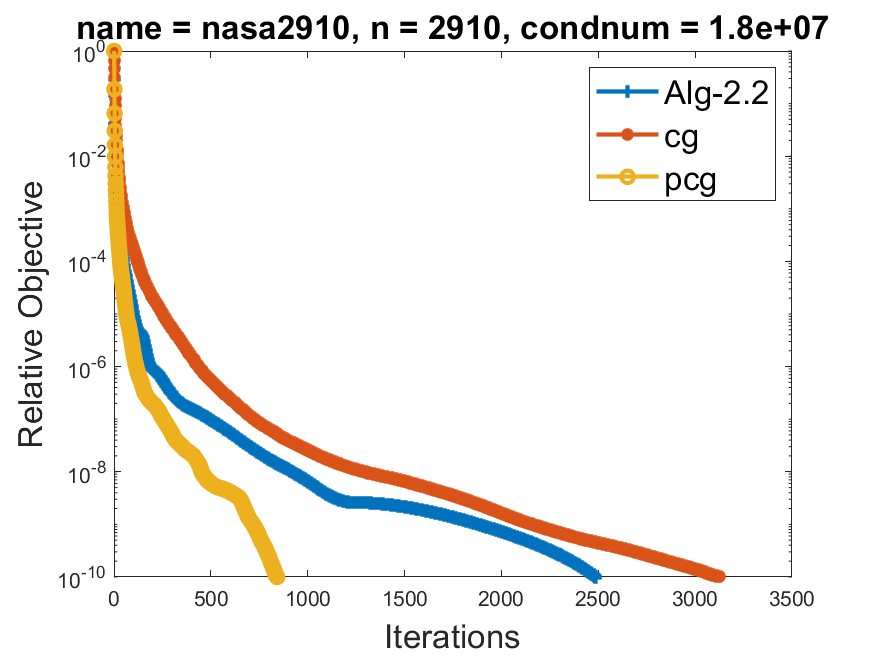}
		\subcaption{nasa2910}
	\end{subfigure}
	\hfill 
	\begin{subfigure}[b]{0.3\textwidth}
		\centering
		\includegraphics[width=\textwidth]{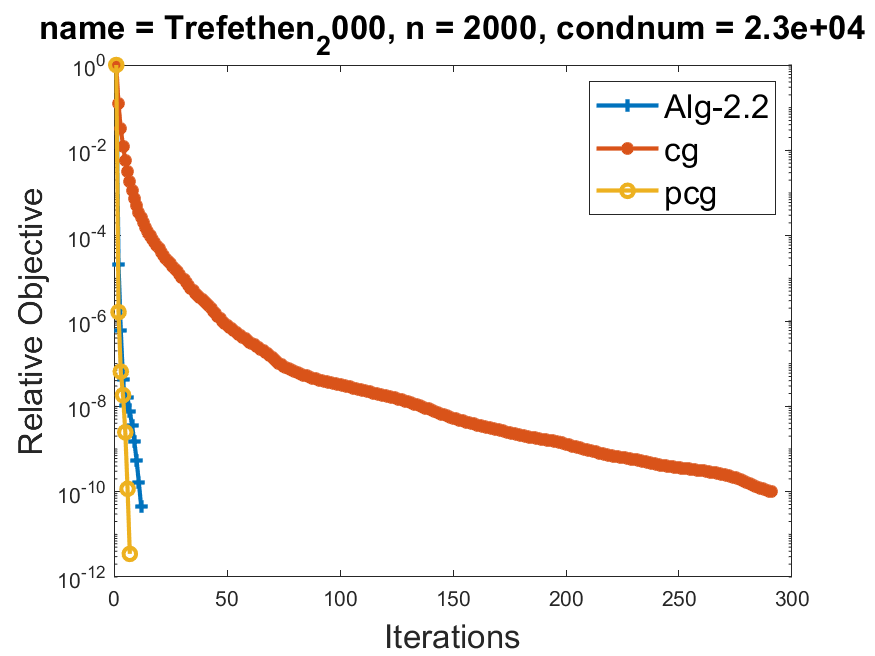}
		\subcaption{Trefethen-2000}
	\end{subfigure}
	\caption{Computational results for {symmetric positive definite} systems collected from the SuiteSparse collection.}
	\label{fig-suitesparse}
\end{figure}

\subsection{Scalability in parallel computing environments} \label{subsection-speedup}

We implement a distributed version of Algorithm \ref{alg-jacobi} using Python and petsc4py \cite{dalcinpazklercosimo2011}. We always consider large sparse matrices in parallel computing environments and assume $p$ processors are available for computation. Matrix $Q$ and vectors $x, y$ are partitioned by rows, and Matrix $J$ is constructed locally and independently. All matrices are stored in CSR format, i.e., the \textit{AIJ} format in petsc4py. The application of the inverse of $J$ in step $4$ of Algorithm \ref{alg-jacobi}
is conducted via $p$ local and independent KSP solvers \cite{freund1992iterative, van2003iterative, saad2001parallel}. The matrix-vector multiplication $Qy$ is well-studied and evaluated via a built-in method \textit{mult} in petsc4py. The parallelization of other components of Algorithm \ref{alg-jacobi} is straightforward.

We test the scalability of the distributed algorithm {with a diagonal preconditioner} on the following six sparse matrices: \texttt{G3\_circuit}, \texttt{bone010}, \texttt{apache2}, \texttt{audikw\_1}, \texttt{StocF-1465}, and \texttt{thermal2}, which are SPD matrices from the {SuiteSparse Matrix Collection}.
Various properties of these matrices used in the test are presented in Table \ref{tb:matrixproperties}. Load imbalance is defined as the ratio of the maximum number of nonzeros assigned to a processor to the average number of nonzeros in each processor: $\dfrac{p \cdot \max_{i}\{{\rm nnz}(Q^{i,:})\}}{{\rm nnz}(Q)}$.

\begin{table}[htb!]
	\centering
	\scalebox{1.0}{
		\begin{tabular}{|c|c|c|c|c|}
			\hline
			sparse matrix & n & avg degree & \#nonzeros & load imb.\\
			\hline
			\texttt{G3\_circuit} & 1.58M & 4.83 & 7.66M & 1.07 \\
			\texttt{bone010} & 987K & 48.5 & 47.9M & 1.02 \\
			\texttt{apache2} & 715K & 6.74 & 4.82M & 1.01 \\
			\texttt{audikw\_1} & 944K & 82.3 & 77.7M & 2.24 \\
			\texttt{StocF-1465} & 1.47M & 14.3 & 21.0M & 1.03 \\
			\texttt{thermal2} & 1.23M & 6.99 & 8.58M & 1.00 \\
			\hline
			
		\end{tabular}
	}
    \caption{Properties of matrices used in scalability test. The values under the ``load imb." column present the load imbalance in terms of the sparse matrix elements for 64 processors.}
	\label{tb:matrixproperties}
\end{table}

Figure \ref{fig-scalability} summarizes the scalability test results where the distributed algorithm is run for 100 iterations. {The cost of the construction of the matrix $J$ is constant regardless of the number of iterations, and the construction cost is minor when the number of iterations is large. Solving local linear equations defined by $J_{i}$ (``solve J'') and  matrix-vector multiplication (``matvec'') are scalable. The restarting technique requires a dot product (``dot'') which is known to be not scalable. Fortunately, the dot product only needs to be computed periodically but not at each every iteration. Consequently, the entire distributed algorithm (``totals'') is scalable until the dot product becomes a scalability bottleneck. Note that this does not indicate the proposed algorithm is not competitive as dot products widely present in competitors including PCG and CG. In fact, Figure \ref{fig-comppcg} shows that the proposed distributed algorithm (``p acc-jacobi'') admits excellent performance in terms of scalability, which is comparable to the parallel PCG (``KSPCG'').}

\begin{figure}[htb!]
	\begin{center}
		\begin{tabular}{ccc}
			\includegraphics[width=0.3\textwidth]{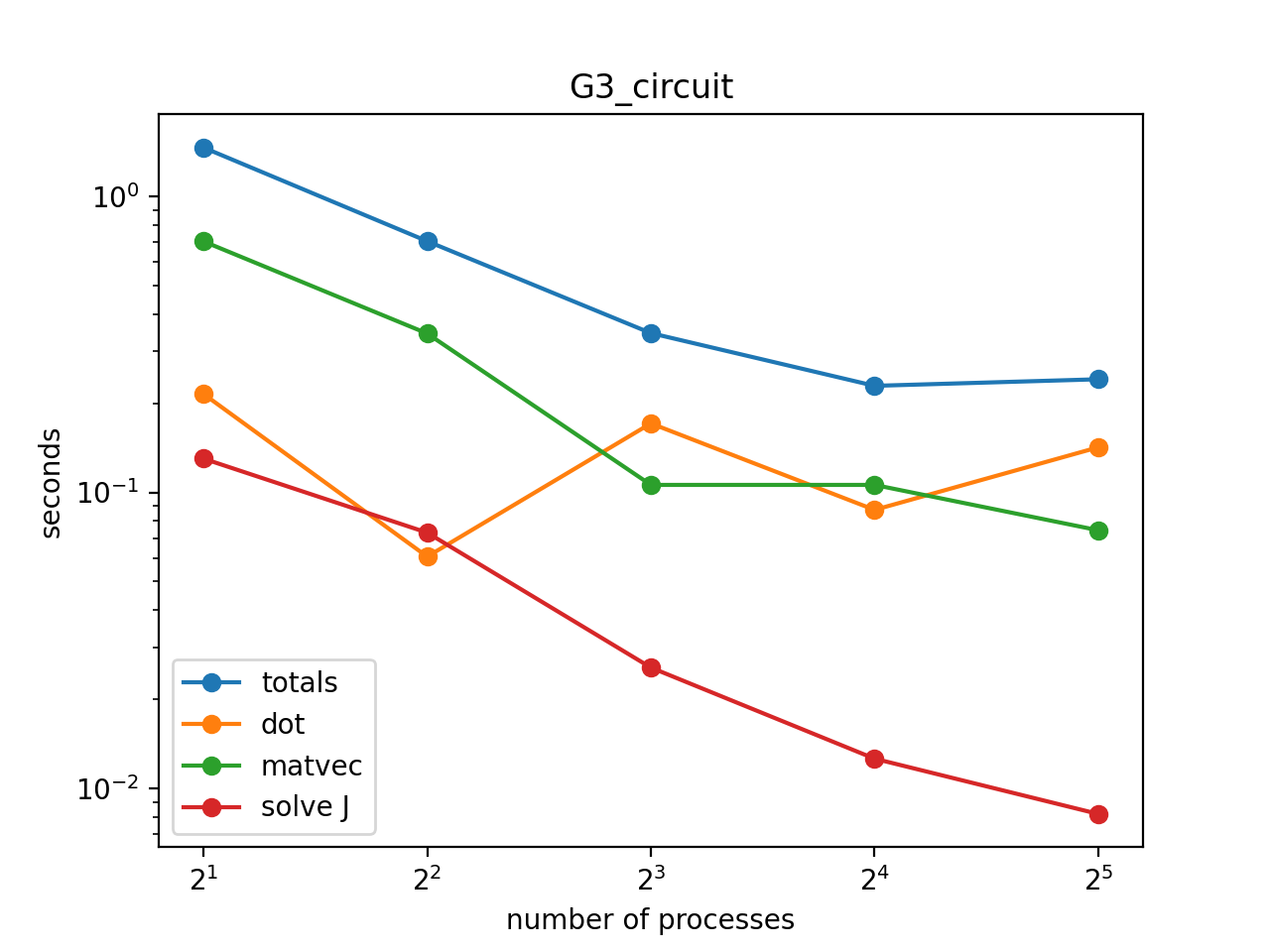}& 
			\includegraphics[width=0.3\textwidth]{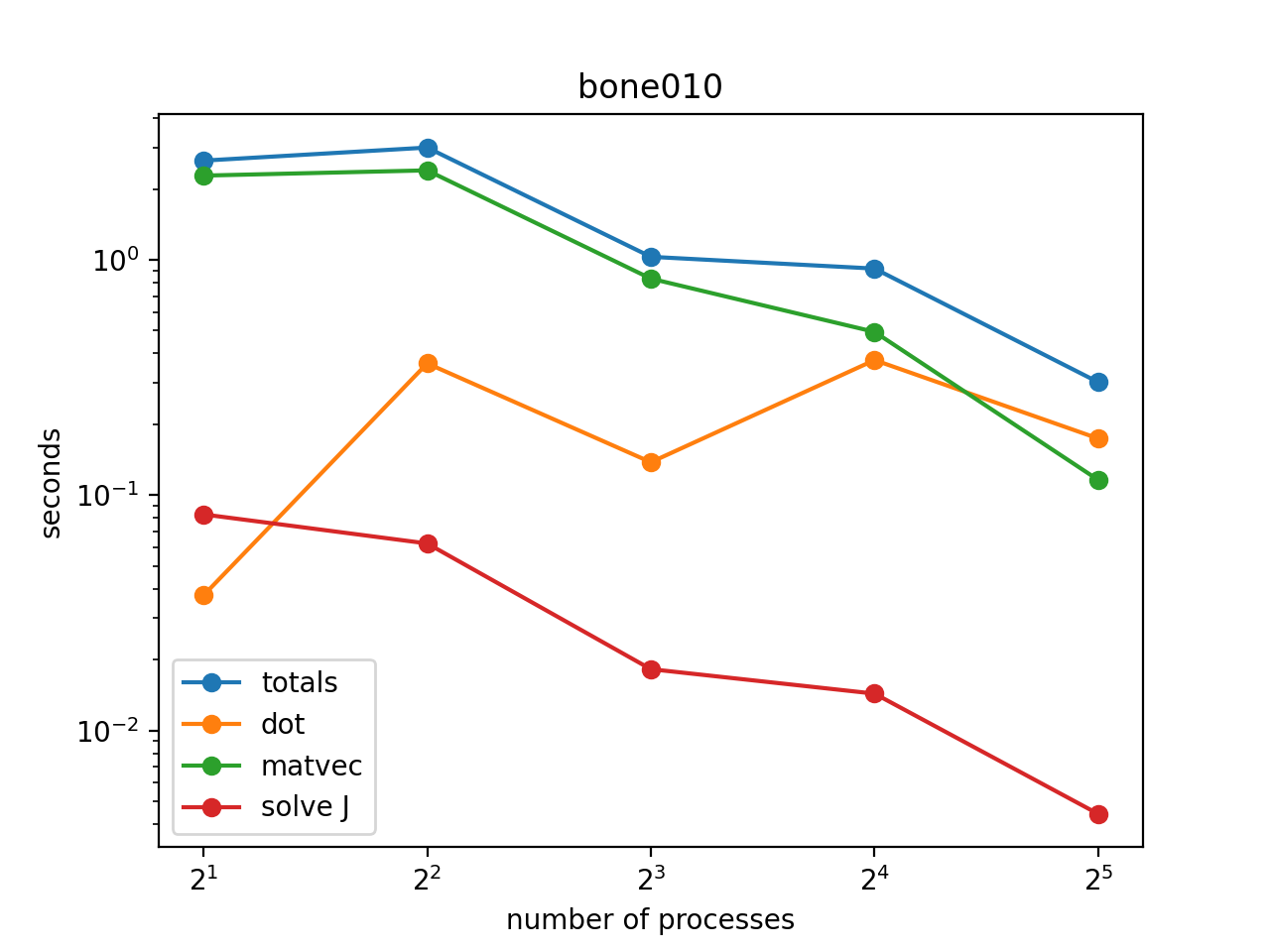} &
			\includegraphics[width=0.3\textwidth]{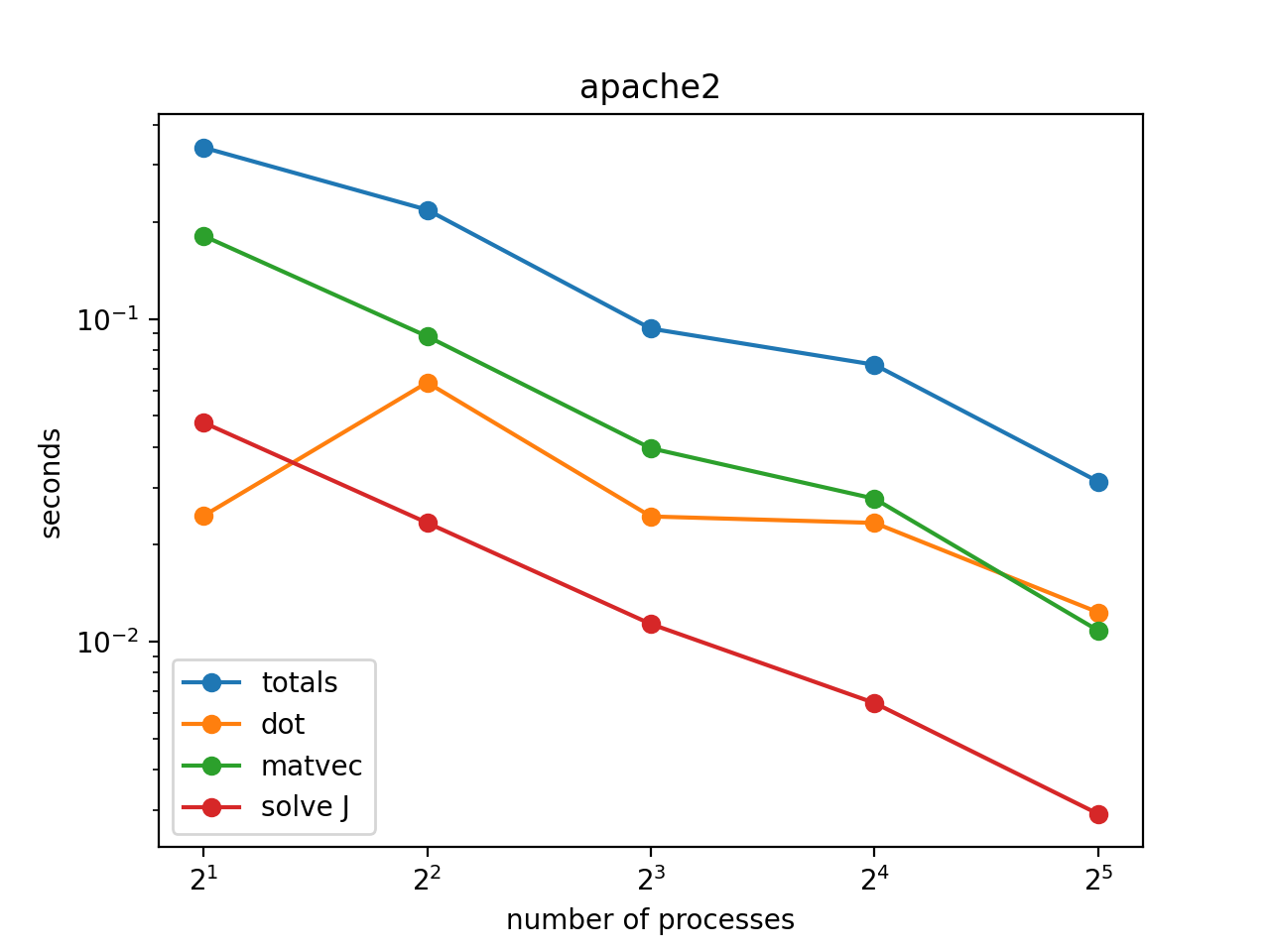} \\ 
			\includegraphics[width=0.3\textwidth]{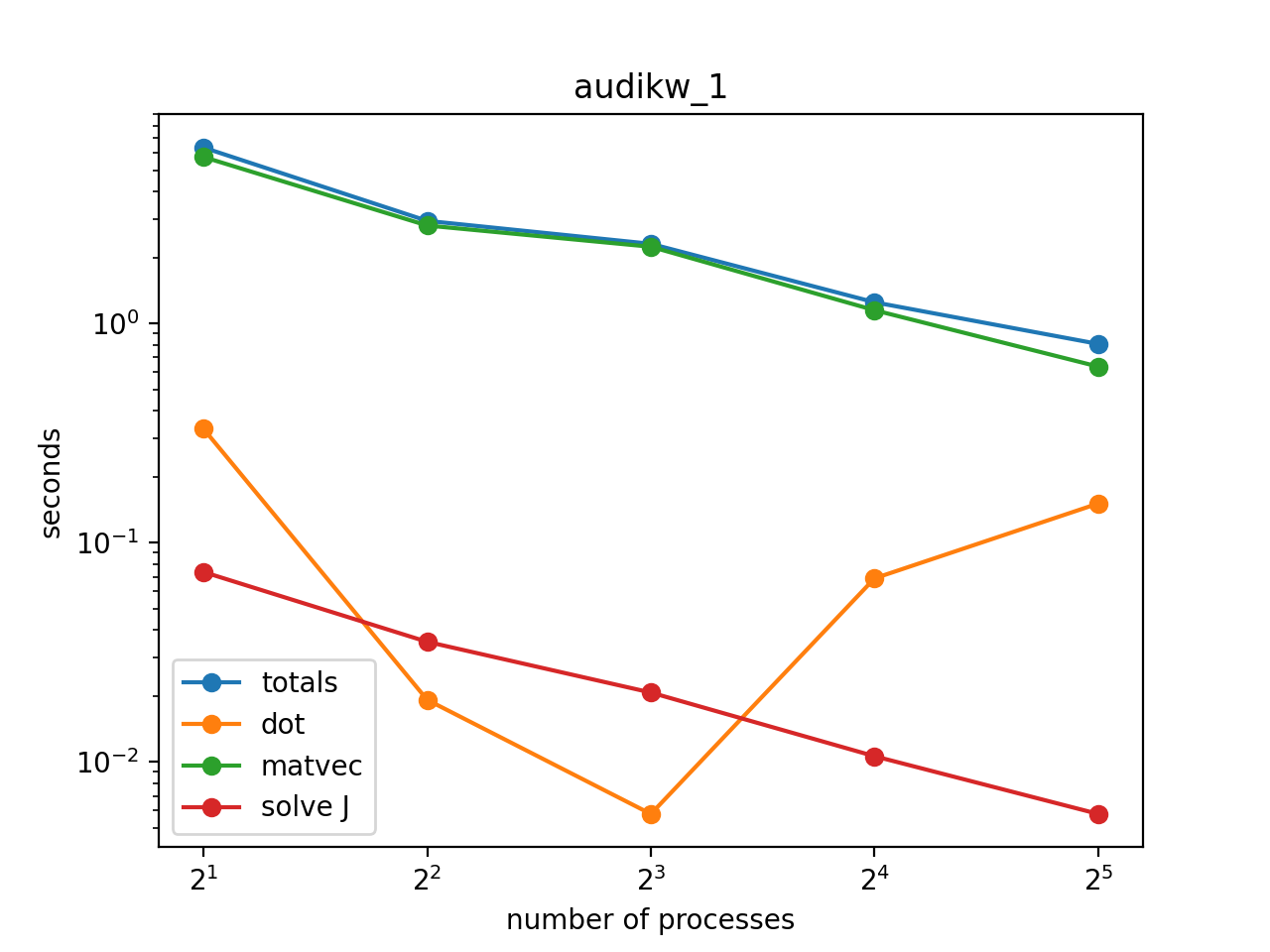} &
			\includegraphics[width=0.3\textwidth]{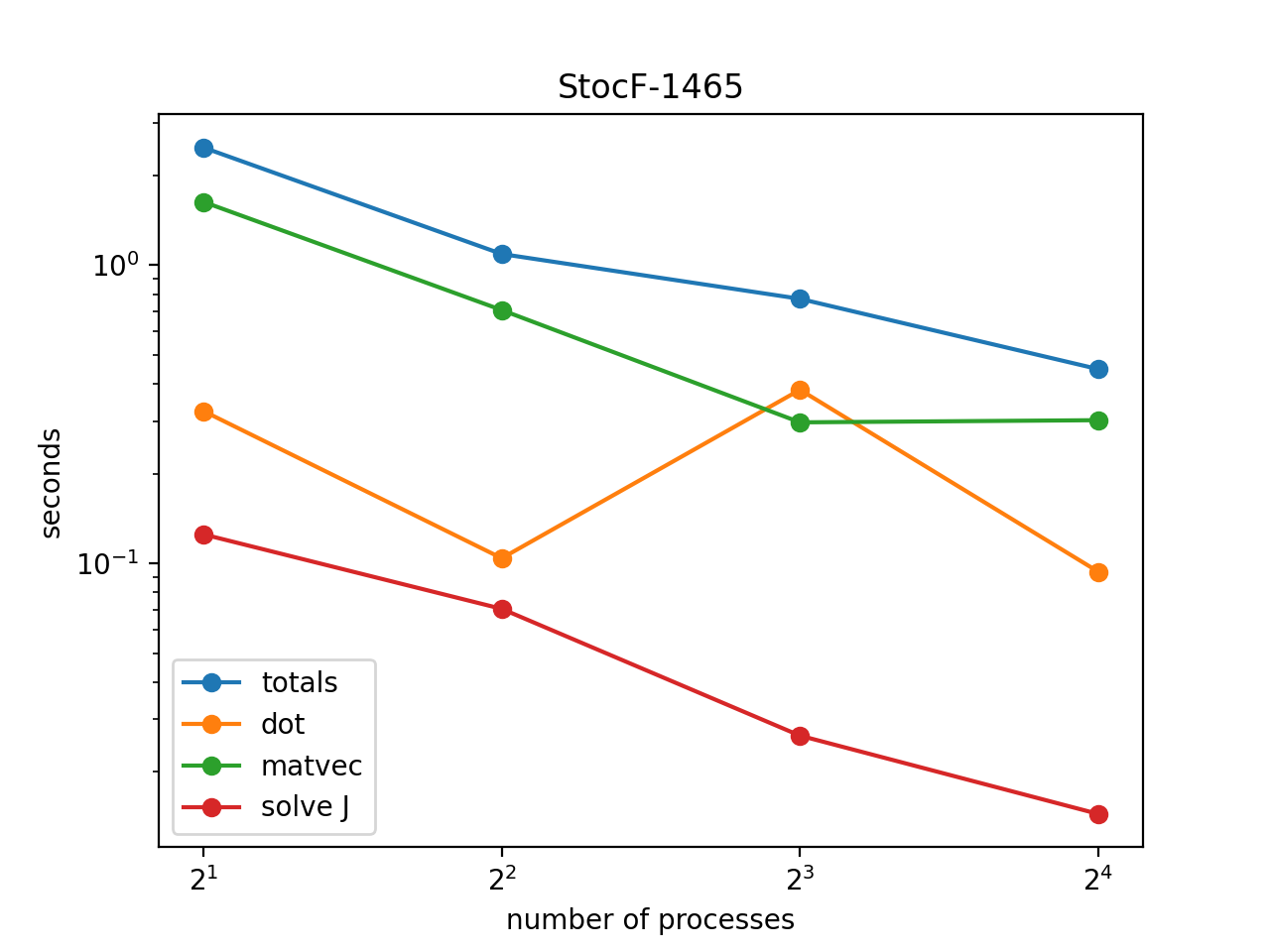}& 
			\includegraphics[width=0.3\textwidth]{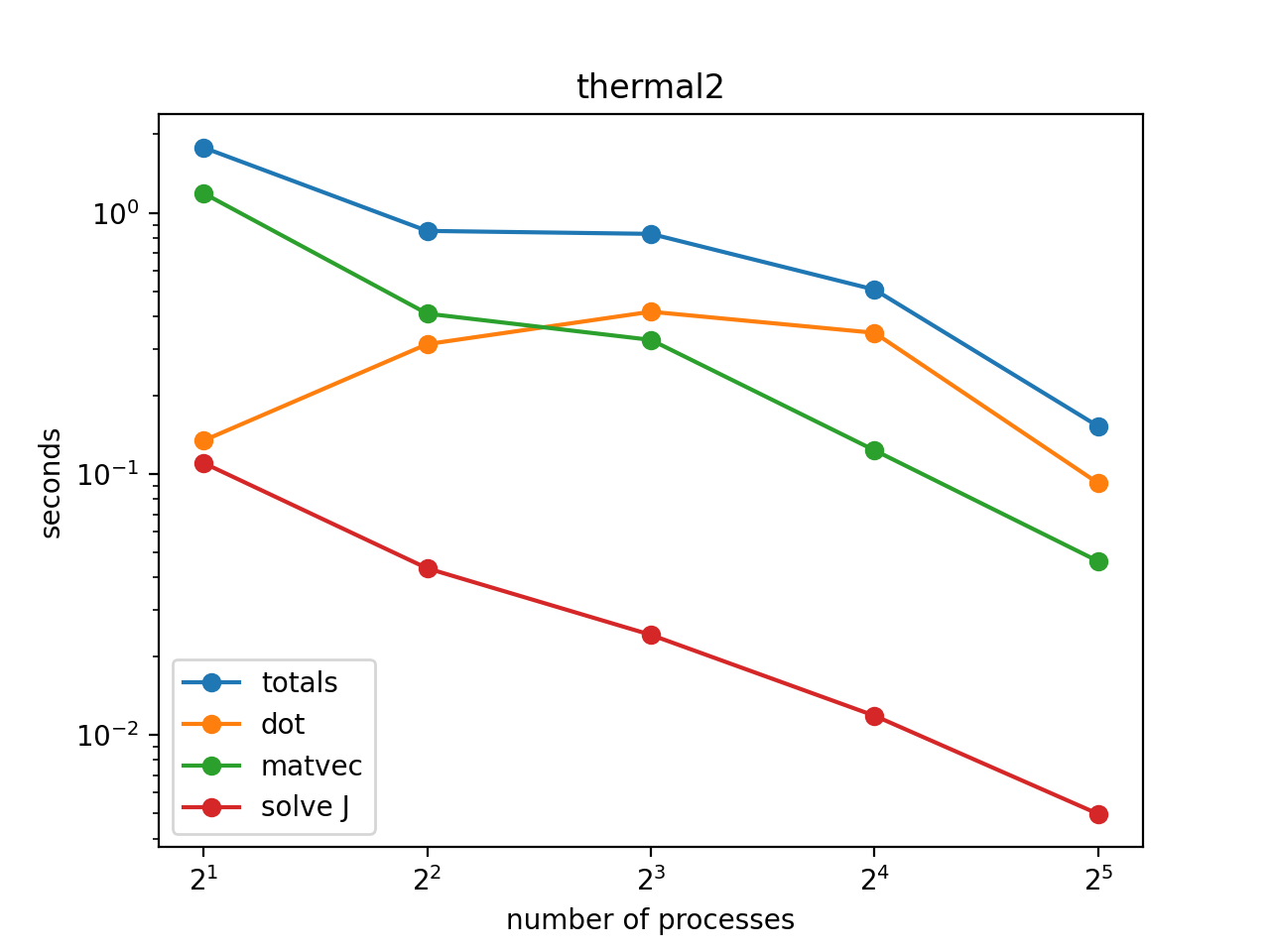} 
		\end{tabular}
	\end{center}
	\caption{Scalability of the distributed version of Algorithm \ref{alg-jacobi} against up to 32 processors. The distributed algorithm is run for $100$ iterations.}
	\label{fig-scalability}
\end{figure}

\begin{figure}[htb!]
	\begin{center}
		\begin{tabular}{ccc}
			\includegraphics[width=0.3\textwidth]{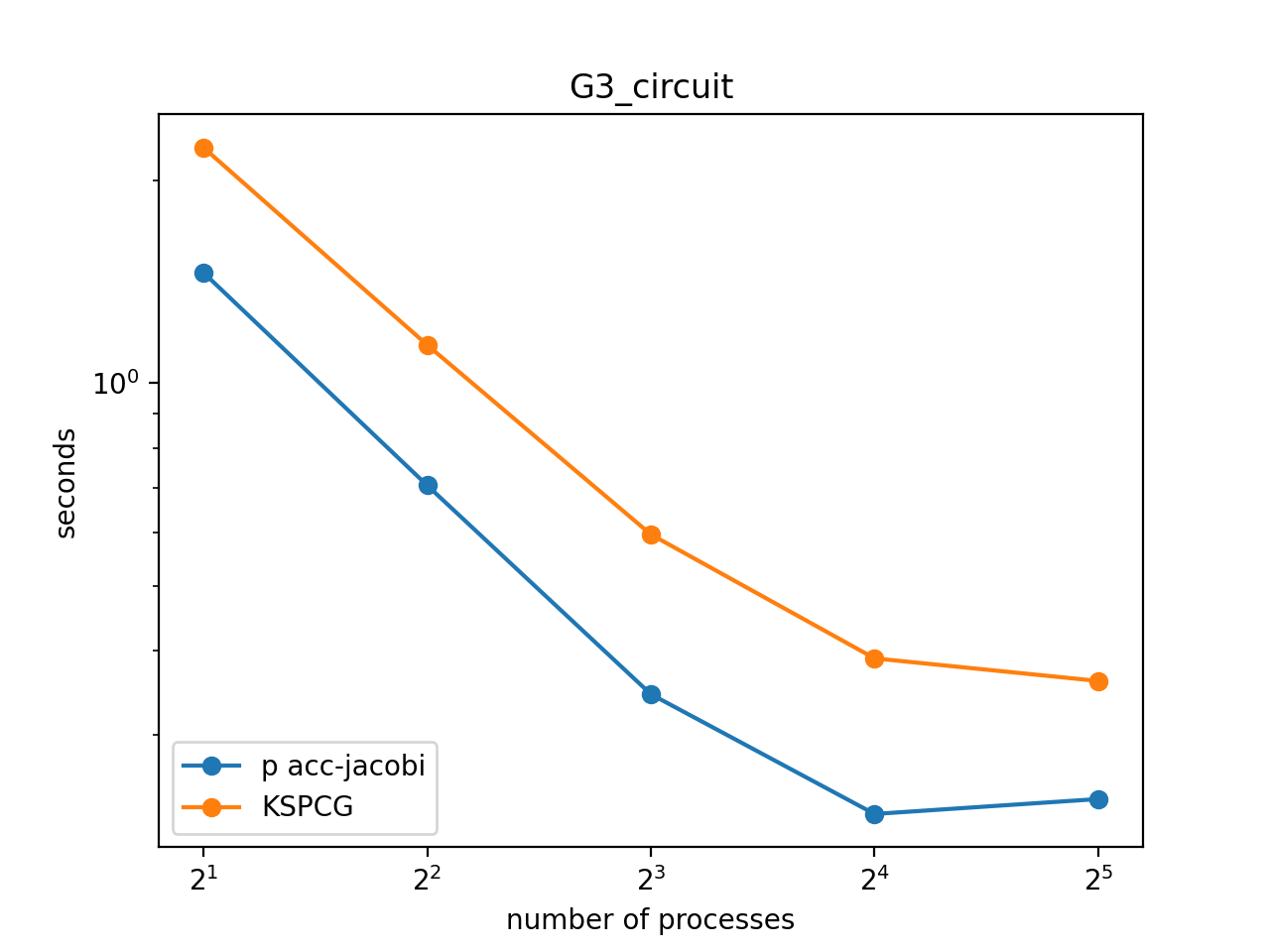}& 
			\includegraphics[width=0.3\textwidth]{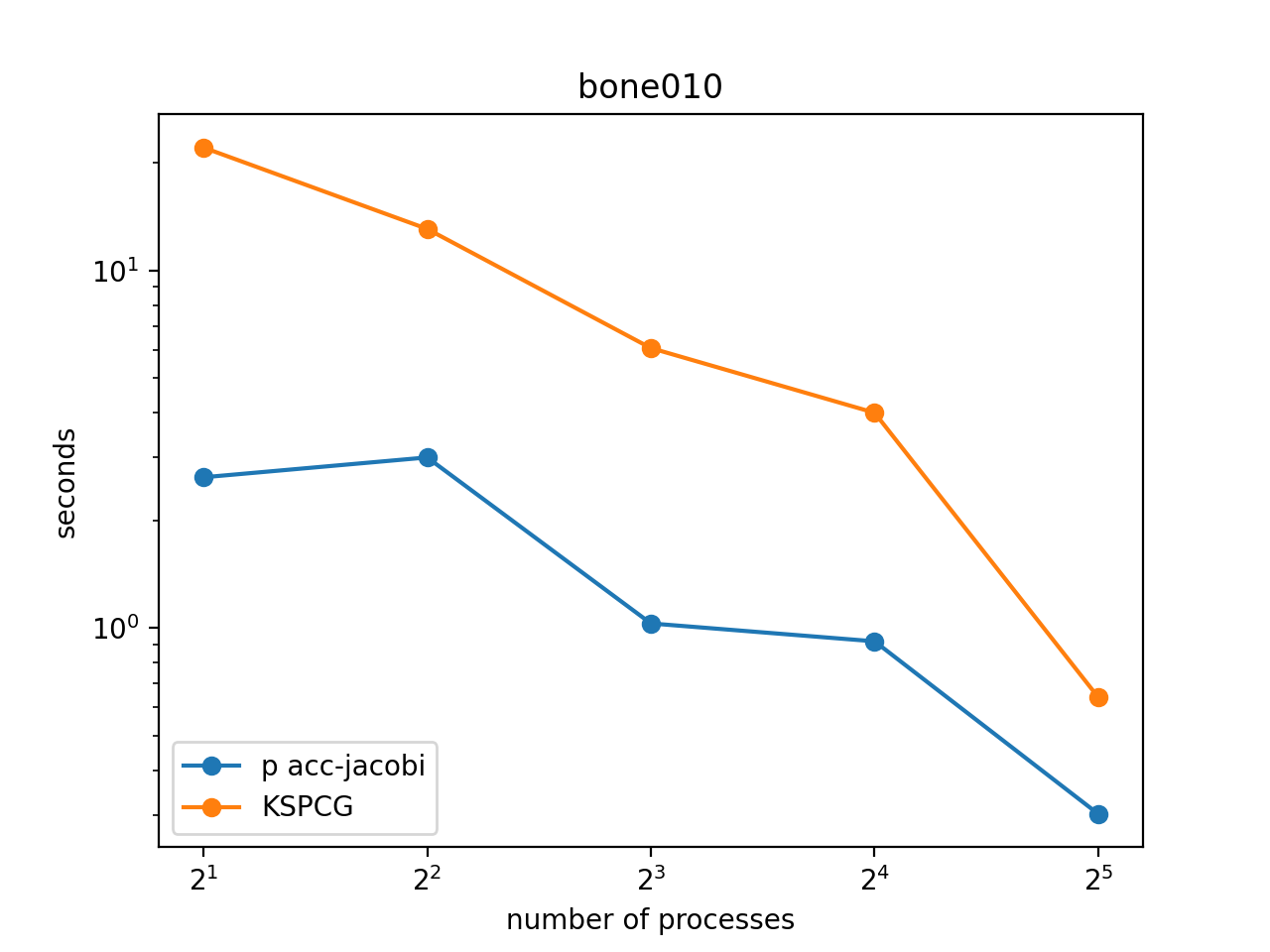} &
			\includegraphics[width=0.3\textwidth]{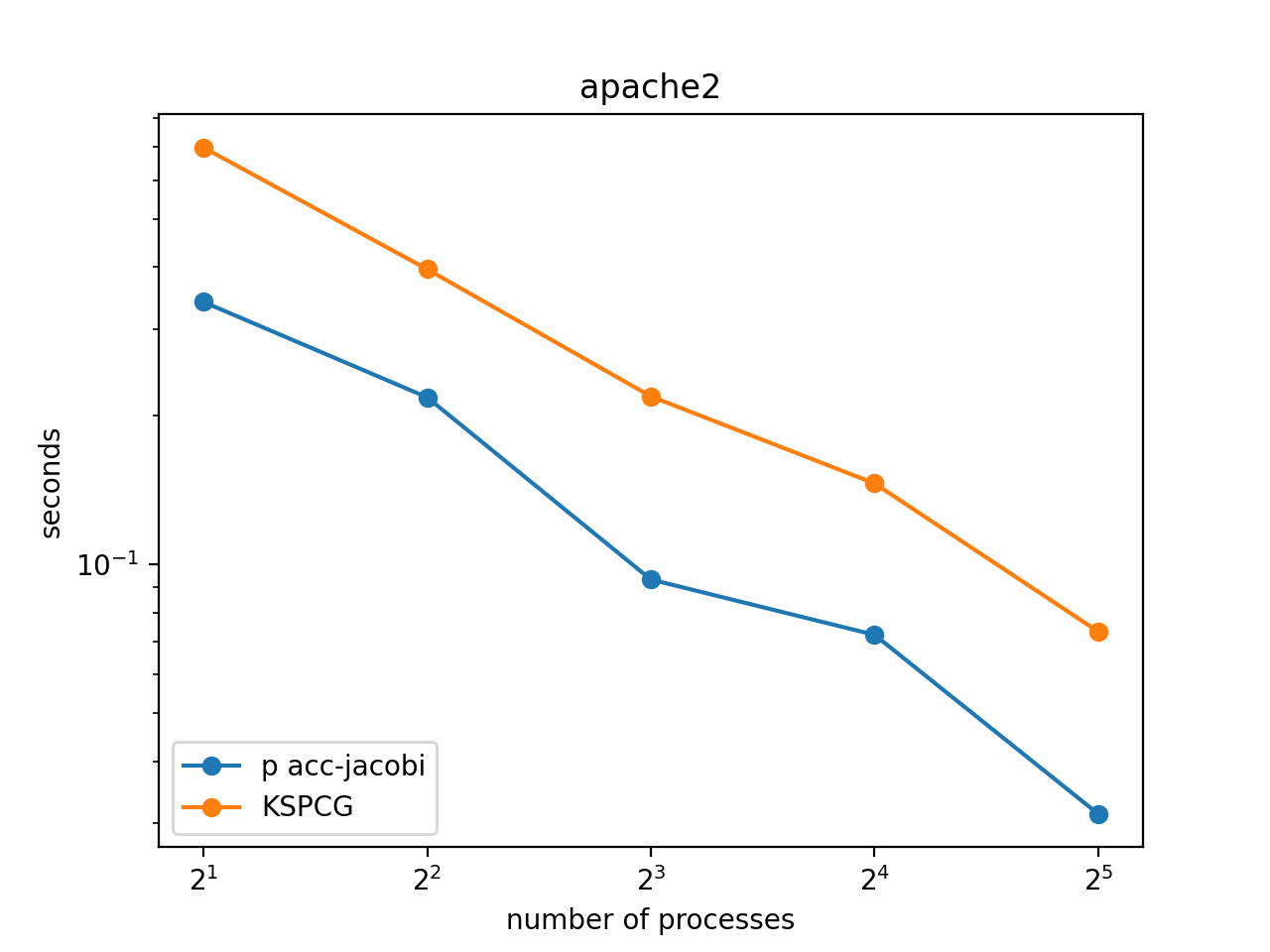} \\ 
			\includegraphics[width=0.3\textwidth]{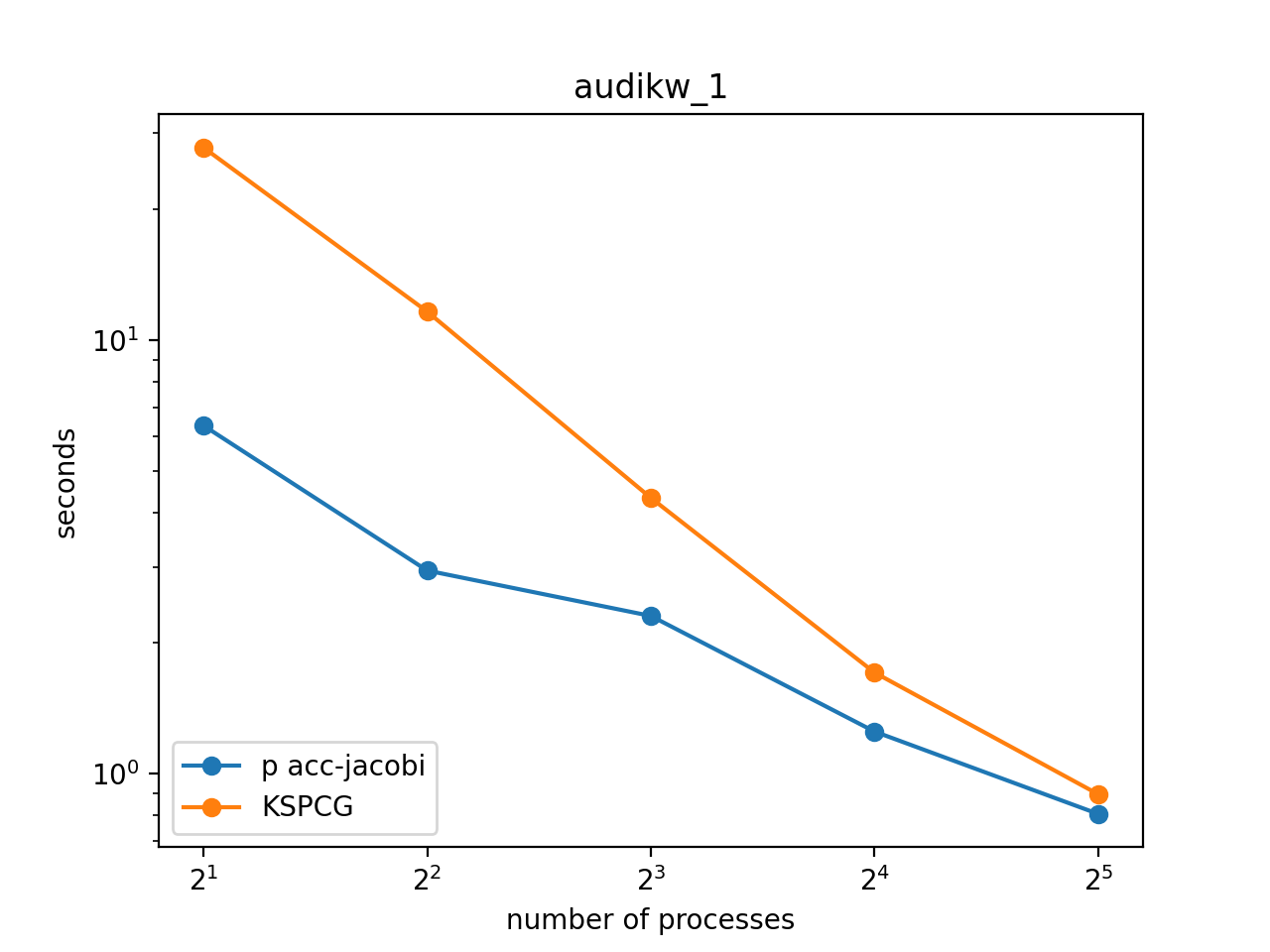} &
			\includegraphics[width=0.3\textwidth]{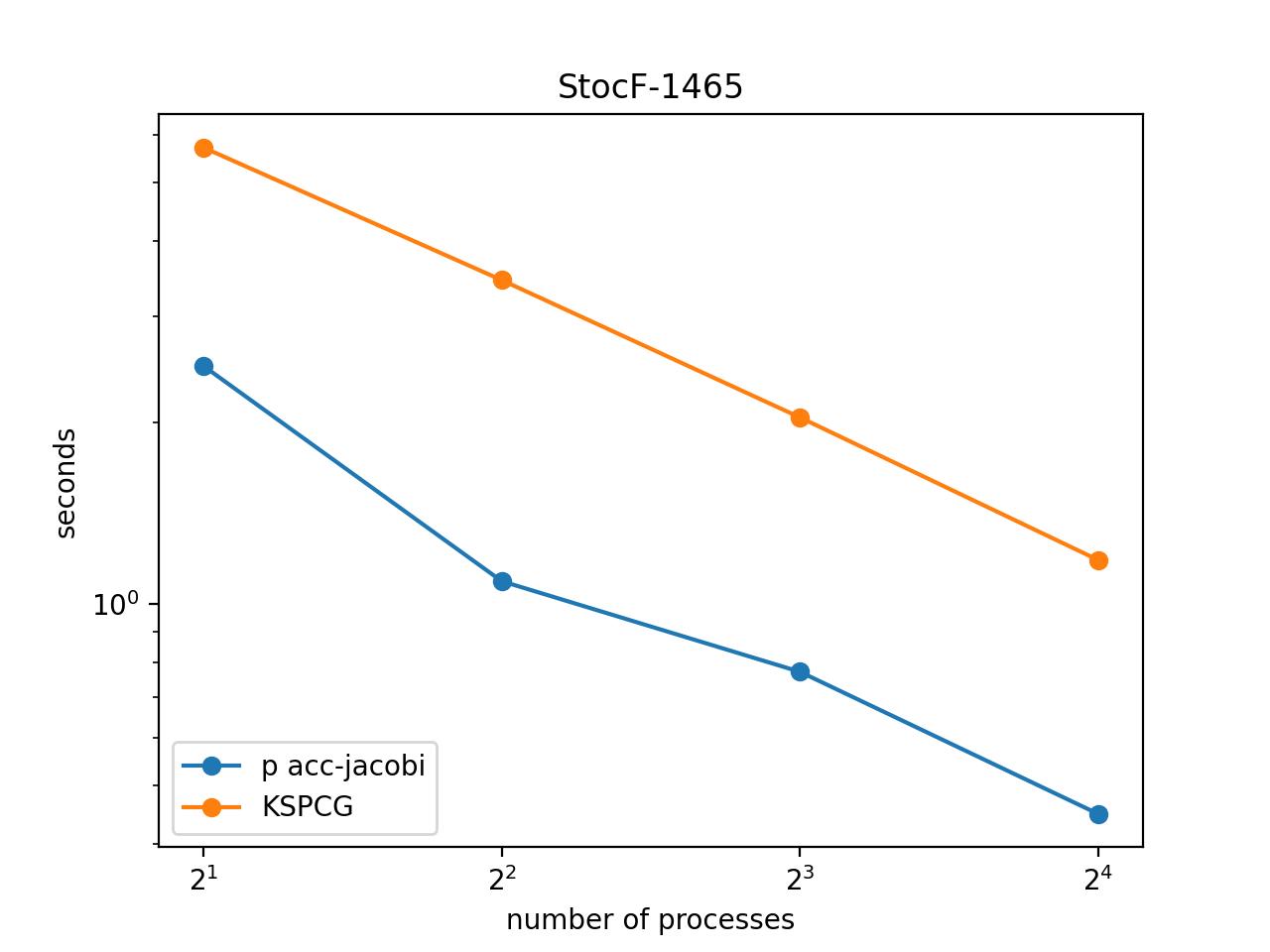}& 
			\includegraphics[width=0.3\textwidth]{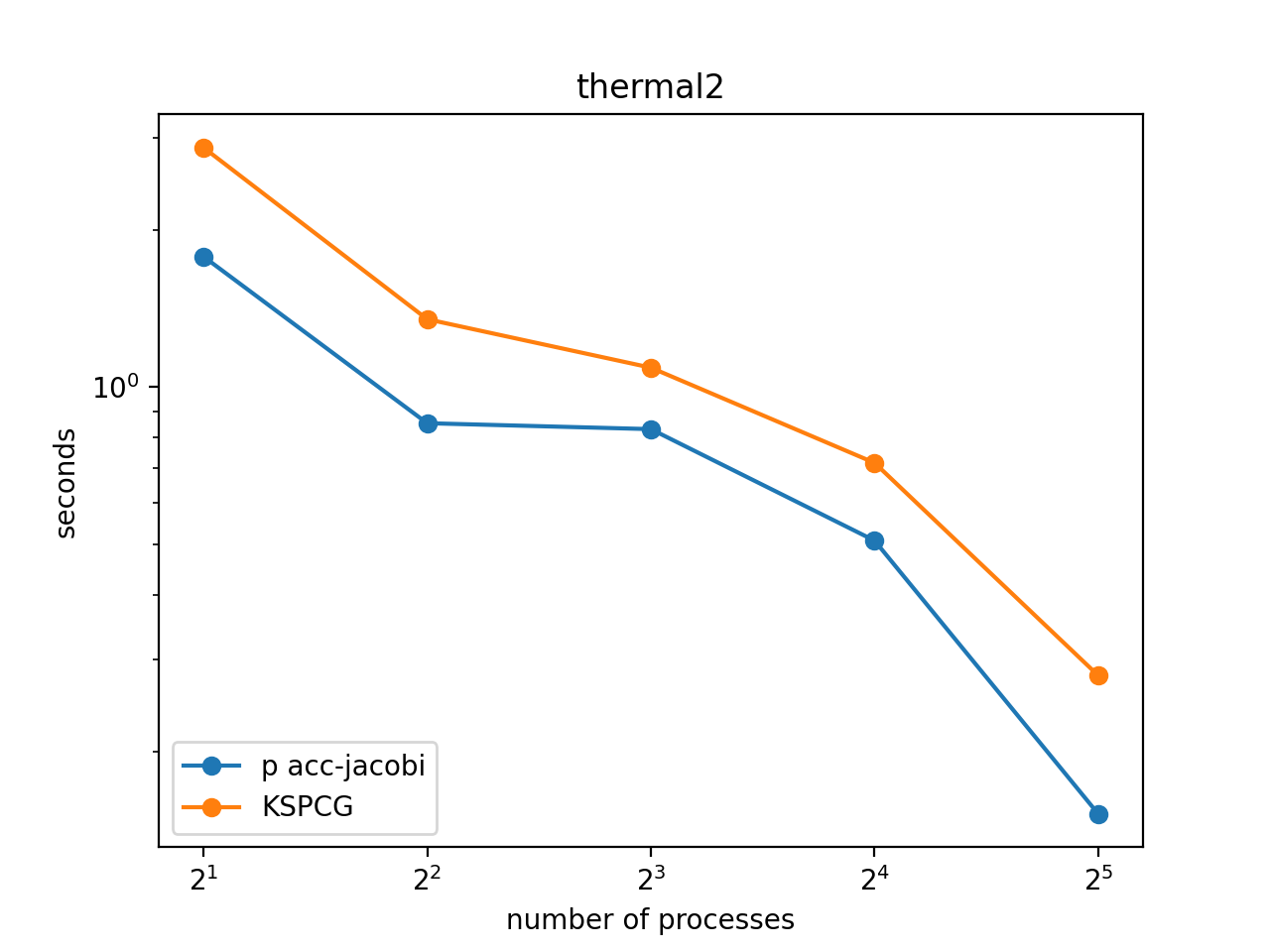} 
		\end{tabular}
	\end{center}
	\caption{Comparisons between the proposed distributed algorithm and pcg with the BJACOBI preconditioner. Both algorithms run for $100$ iterations.}
	\label{fig-comppcg}
\end{figure}

\section{Conclusions} \label{section-conclusions}

In this paper, we have considered solving large-scale symmetric positive semidefinite (SPSD) systems of linear equations via Jacobi-type iterative methods, which improve the classical Jacobi-type iterative methods by applying Nesterov's acceleration technique. The proposed method has been shown to be convergent for any consistent SPSD systems with an $O\left(\frac{1}{t^2}\right)$ convergence rate in terms of objective function values associated with the underlying linear system, provided that the proximal term in the Nesterov's framework is induced by a symmetric positive semidefinite matrix. We have also proposed and analyzed a restarted version of the former method and shown that it has an $O\left(\frac{(\log_2(t))^2}{t^2}\right)$ worst-case convergence rate, under the condition that the matrix $Q$ is positive definite. To evaluate the practical performance of the proposed methods, we have conducted extensive numerical experiments. Our numerical results indicate that the proposed method outperforms the classical Jacobi-type methods and the conjugate gradient method, and shares a comparable performance as the diagonally preconditioned conjugate gradient method. The speed-up tests also provide strong evidence demonstrating that the proposed method is highly scalable. 

\section*{Acknowledgments}
Haizhao Yang was partially supported by the US National Science Foundation under awards DMS-2244988, DMS-2206333, the Office of Naval Research Award N00014-23-1-2007, and the DARPA D24AP00325-00.

{We thank the editor and three anonymous reviewers for providing valuable suggestions, which have helped to improve the quality of the paper.}

\bibliographystyle{siamplain}
\bibliography{references}

\begin{thebibliography}{10}

\bibitem{petsc-user-ref}
{\sc S.~Balay, S.~Abhyankar, M.~F. Adams, S.~Benson, J.~Brown, P.~Brune,
  K.~Buschelman, E.~Constantinescu, L.~Dalcin, A.~Dener, V.~Eijkhout,
  J.~Faibussowitsch, W.~D. Gropp, V.~Hapla, T.~Isaac, P.~Jolivet, D.~Karpeev,
  D.~Kaushik, M.~G. Knepley, F.~Kong, S.~Kruger, D.~A. May, L.~C. McInnes,
  R.~T. Mills, L.~Mitchell, T.~Munson, J.~E. Roman, K.~Rupp, P.~Sanan,
  J.~Sarich, B.~F. Smith, S.~Zampini, H.~Zhang, H.~Zhang, and J.~Zhang}, {\em
  {PETSc/TAO} users manual}, Tech. Report ANL-21/39 - Revision 3.20, Argonne
  National Laboratory, 2023.

\bibitem{petsc-efficient}
{\sc S.~Balay, W.~D. Gropp, L.~C. McInnes, and B.~F. Smith}, {\em Efficient
  management of parallelism in object oriented numerical software libraries},
  in Modern Software Tools in Scientific Computing, E.~Arge, A.~M. Bruaset, and
  H.~P. Langtangen, eds., Birkh{\"{a}}user Press, 1997, pp.~163--202.

\bibitem{beck2009fast}
{\sc A.~Beck and M.~Teboulle}, {\em A fast iterative shrinkage-thresholding
  algorithm for linear inverse problems}, SIAM Journal on Imaging Sciences, 2
  (2009), pp.~183--202.

\bibitem{bertsekas2015parallel}
{\sc D.~Bertsekas and J.~Tsitsiklis}, {\em Parallel and Distributed
  Computation: Numerical Methods}, Athena Scientific, 2015.

\bibitem{bonnans1995family}
{\sc J.~F. Bonnans, J.~C. Gilbert, C.~Lemar{\'e}chal, and C.~A.
  Sagastiz{\'a}bal}, {\em A family of variable metric proximal methods},
  Mathematical Programming, 68 (1995), pp.~15--47.

\bibitem{boyd2004convex}
{\sc S.~P. Boyd and L.~Vandenberghe}, {\em Convex Optimization}, Cambridge
  University Press, 2004.

\bibitem{cantu2000efficient}
{\sc E.~Cantu-Paz}, {\em Efficient and accurate parallel genetic algorithms},
  vol.~1, Springer Science \& Business Media, 2000.

\bibitem{cantu1998survey}
{\sc E.~Cant{\'u}-Paz et~al.}, {\em A survey of parallel genetic algorithms},
  Calculateurs Paralleles, Reseaux et Systems Repartis, 10 (1998),
  pp.~141--171.

\bibitem{chambolle2015convergence}
{\sc A.~Chambolle and C.~H. Dossal}, {\em On the convergence of the iterates of
  ``{FISTA}''}, Journal of Optimization Theory and Applications, 166 (2015),
  p.~25.

\bibitem{chung1997spectral}
{\sc F.~R. Chung}, {\em Spectral graph theory}, vol.~92, American Mathematical
  Soc., 1997.

\bibitem{dalcinpazklercosimo2011}
{\sc L.~D. Dalcin, R.~R. Paz, P.~A. Kler, and A.~Cosimo}, {\em Parallel
  distributed computing using {P}ython}, Advances in Water Resources, 34
  (2011), pp.~1124 -- 1139.
\newblock New Computational Methods and Software Tools.

\bibitem{fogarty1991implementing}
{\sc T.~C. Fogarty and R.~Huang}, {\em Implementing the genetic algorithm on
  transputer based parallel processing systems}, in Parallel Problem Solving
  from Nature: 1st Workshop, PPSN I Dortmund, FRG, October 1--3, 1990
  Proceedings 1, Springer, 1991, pp.~145--149.

\bibitem{freund1992iterative}
{\sc R.~W. Freund, G.~H. Golub, and N.~M. Nachtigal}, {\em Iterative solution
  of linear systems}, Acta Numerica, 1 (1992), pp.~57--100.

\bibitem{golub2013matrix}
{\sc G.~H. Golub and C.~F. Van~Loan}, {\em Matrix Computations}, JHU press,
  2013.

\bibitem{greenbaum1997iterative}
{\sc A.~Greenbaum}, {\em Iterative Methods for Solving Linear Systems}, SIAM,
  1997.

\bibitem{harada2020parallel}
{\sc T.~Harada and E.~Alba}, {\em Parallel genetic algorithms: a useful
  survey}, ACM Computing Surveys (CSUR), 53 (2020), pp.~1--39.

\bibitem{hinze2008optimization}
{\sc M.~Hinze, R.~Pinnau, M.~Ulbrich, and S.~Ulbrich}, {\em Optimization with
  {PDE} constraints}, vol.~23, Springer Science \& Business Media, 2008.

\bibitem{imakura2012auto}
{\sc A.~Imakura, T.~Sakurai, K.~Sumiyoshi, and H.~Matsufuru}, {\em An
  auto-tuning technique of the weighted {J}acobi-type iteration used for
  preconditioners of {K}rylov subspace methods}, in 2012 IEEE 6th International
  Symposium on Embedded Multicore SoCs, IEEE, 2012, pp.~183--190.

\bibitem{incropera1996fundamentals}
{\sc F.~P. Incropera, D.~P. DeWitt, T.~L. Bergman, A.~S. Lavine, et~al.}, {\em
  Fundamentals of Heat and Mass Transfer}, vol.~6, Wiley New York, 1996.

\bibitem{ito2017unified}
{\sc N.~Ito, A.~Takeda, and K.-C. Toh}, {\em A unified formulation and fast
  accelerated proximal gradient method for classification}, The Journal of
  Machine Learning Research, 18 (2017), pp.~510--558.

\bibitem{jacobi1846leichtes}
{\sc C.~G.~J. Jacobi}, {\em {\"U}ber ein leichtes verfahren die in der theorie
  der s{\"a}cularst{\"o}rungen vorkommenden gleichungen numerisch
  aufzul{\"o}sen*},  (1846).

\bibitem{jensen2024nonsmooth}
{\sc B.~Jensen and T.~Valkonen}, {\em A nonsmooth primal-dual method with
  interwoven {PDE} constraint solver}, Computational Optimization and
  Applications,  (2024), pp.~1--35.

\bibitem{johnson2012numerical}
{\sc C.~Johnson}, {\em Numerical solution of partial differential equations by
  the finite element method}, Courier Corporation, 2012.

\bibitem{karczmarz1937angenaherte}
{\sc S.~Karczmarz}, {\em Angenaherte auflosung von systemen linearer
  glei-chungen}, Bull. Int. Acad. Pol. Sic. Let., Cl. Sci. Math. Nat.,  (1937),
  pp.~355--357.

\bibitem{kimmel2019extensions}
{\sc G.~J. Kimmel and A.~Glatz}, {\em Extensions and analysis of worst-case
  parameter in weighted {J}acobi's method for solving second order implicit
  {PDEs}}, Results in Applied Mathematics, 1 (2019), p.~100003.

\bibitem{lee2022escaping}
{\sc C.-p. Lee, L.~Liang, T.~Tang, and K.-C. Toh}, {\em Accelerating
  nuclear-norm regularized low-rank matrix optimization through
  {Burer-Monteiro} decomposition}, Journal of Machine Learning Research, 25
  (2024), pp.~1--52.

\bibitem{li2019block}
{\sc X.~Li, D.~Sun, and K.-C. Toh}, {\em A block symmetric {Gauss--Seidel}
  decomposition theorem for convex composite quadratic programming and its
  applications}, Mathematical Programming, 175 (2019), pp.~395--418.

\bibitem{li2005overview}
{\sc X.~S. Li}, {\em An overview of {SuperLU}: {A}lgorithms, implementation,
  and user interface}, ACM Transactions on Mathematical Software (TOMS), 31
  (2005), pp.~302--325.

\bibitem{liang2022qppal}
{\sc L.~Liang, X.~Li, D.~Sun, and K.-C. Toh}, {\em {QPPAL}: {A} two-phase
  proximal augmented {L}agrangian method for high-dimensional convex quadratic
  programming problems}, ACM Transactions on Mathematical Software (TOMS), 48
  (2022), pp.~1--27.

\bibitem{monteiro2016adaptive}
{\sc R.~D. Monteiro, C.~Ortiz, and B.~F. Svaiter}, {\em An adaptive accelerated
  first-order method for convex optimization}, Computational Optimization and
  Applications, 64 (2016), pp.~31--73.

\bibitem{nesterov2013gradient}
{\sc Y.~Nesterov}, {\em Gradient methods for minimizing composite functions},
  Mathematical Programming, 140 (2013), pp.~125--161.

\bibitem{nesterov2018lectures}
{\sc Y.~Nesterov et~al.}, {\em Lectures on Convex Optimization}, vol.~137,
  Springer, 2018.

\bibitem{ng2001spectral}
{\sc A.~Ng, M.~Jordan, and Y.~Weiss}, {\em On spectral clustering: Analysis and
  an algorithm}, Advances in Neural Information Processing Systems, 14 (2001).

\bibitem{o2015adaptive}
{\sc B.~O’donoghue and E.~Candes}, {\em Adaptive restart for accelerated
  gradient schemes}, Foundations of Computational Mathematics, 15 (2015),
  pp.~715--732.

\bibitem{pang2023spectral}
{\sc Q.~Pang and H.~Yang}, {\em Spectral clustering via orthogonalization-free
  methods}, arXiv preprint arXiv:2305.10356,  (2023).

\bibitem{pang2024distributed}
{\sc Q.~Pang and H.~Yang}, {\em A distributed block {Chebyshev-Davidson}
  algorithm for parallel spectral clustering}, Journal of Scientific Computing,
  98 (2024), p.~69.

\bibitem{saad2001parallel}
{\sc Y.~Saad}, {\em Parallel iterative methods for sparse linear systems}, in
  Studies in Computational Mathematics, vol.~8, Elsevier, 2001, pp.~423--440.

\bibitem{saad2003iterative}
{\sc Y.~Saad}, {\em Iterative Methods for Sparse Linear Systems}, SIAM, 2003.

\bibitem{sadiku2001electromagnetics}
{\sc M.~N. Sadiku}, {\em Electromagnetics}, 2001.

\bibitem{schenk2001pardiso}
{\sc O.~Schenk, K.~G{\"a}rtner, W.~Fichtner, and A.~Stricker}, {\em Pardiso: a
  high-performance serial and parallel sparse linear solver in semiconductor
  device simulation}, Future Generation Computer Systems, 18 (2001),
  pp.~69--78.

\bibitem{shi2000normalized}
{\sc J.~Shi and J.~Malik}, {\em Normalized cuts and image segmentation}, IEEE
  Transactions on Pattern Analysis and Machine Intelligence, 22 (2000),
  pp.~888--905.

\bibitem{strohmer2009randomized}
{\sc T.~Strohmer and R.~Vershynin}, {\em A randomized kaczmarz algorithm with
  exponential convergence}, Journal of Fourier Analysis and Applications, 15
  (2009), pp.~262--278.

\bibitem{su2014differential}
{\sc W.~Su, S.~Boyd, and E.~Candes}, {\em A differential equation for modeling
  {N}esterov’s accelerated gradient method: theory and insights}, Advances in
  Neural Information Processing Systems, 27 (2014).

\bibitem{tanese1987parallel}
{\sc R.~Tanese}, {\em Parallel genetic algorithm for a hypercube}, in Genetic
  algorithms and their applications: proceedings of the second International
  Conference on Genetic Algorithms: July 28-31, 1987 at the Massachusetts
  Institute of Technology, Cambridge, MA, 1987.

\bibitem{tu2022linear}
{\sc Y.~Tu, Q.~Pang, H.~Yang, and Z.~Xu}, {\em Linear-scaling selected
  inversion based on hierarchical interpolative factorization for self
  {G}reen's function for modified {P}oisson-{B}oltzmann equation in two
  dimensions}, Journal of Computational Physics, 461 (2022), p.~110893.

\bibitem{valkonen2020inertial}
{\sc T.~Valkonen}, {\em Inertial, corrected, primal-dual proximal splitting},
  SIAM Journal on Optimization, 30 (2020), pp.~1391--1420.

\bibitem{van2003iterative}
{\sc H.~A. Van~der Vorst}, {\em Iterative {K}rylov Methods for Large Linear
  Systems}, no.~13, Cambridge University Press, 2003.

\bibitem{yang2023inexact}
{\sc H.~Yang, L.~Liang, L.~Carlone, and K.-C. Toh}, {\em An inexact projected
  gradient method with rounding and lifting by nonlinear programming for
  solving rank-one semidefinite relaxation of polynomial optimization},
  Mathematical Programming, 201 (2023), pp.~409--472.

\bibitem{young2014iterative}
{\sc D.~M. Young}, {\em Iterative Solution of Large Linear Systems}, Elsevier,
  2014.

\bibitem{zhang2009cuda}
{\sc Z.~Zhang, Q.~Miao, and Y.~Wang}, {\em {CUDA}-based {J}acobi's iterative
  method}, in 2009 International Forum on Computer Science-Technology and
  Applications, vol.~1, IEEE, 2009, pp.~259--262.

\bibitem{zienkiewicz2005finite}
{\sc O.~C. Zienkiewicz and R.~L. Taylor}, {\em The finite element method for
  solid and structural mechanics}, Elsevier, 2005.

\end{thebibliography}
\end{document}